\documentclass[11pt,oneside,leqno]{amsart}
\usepackage[centertags]{amsmath}
\usepackage{amsfonts}
\usepackage{amsmath,amssymb}
\usepackage{latexsym}

\newcommand\g{{\mathfrak{g}}}
\newcommand\h{{\mathfrak{h}}}

\newcommand\ad{{\rm ad}}

\newtheorem{theorem}{Theorem}[section]
\newtheorem{proposition}[theorem]{Proposition}
\newtheorem{corollary}[theorem]{Corollary}
\newtheorem{lemma}[theorem]{Lemma}
\newtheorem{definition}[theorem]{Definition}
\newtheorem{example}[theorem]{Example}
\newtheorem{remark}[theorem]{Remark}

\renewcommand{\theequation }{\thesection .\arabic{theequation}}

\begin{document}
\title[Cosymplectic and $\alpha$-cosymplectic Lie algebras]{Cosymplectic and $\alpha$-cosymplectic  Lie algebras}
\author{Giovanni Calvaruso and Antonella Perrone}
\date{}

\address{Dipartimento di Matematica e Fisica \lq\lq E. De Giorgi\rq\rq \\
Universit\`a del Salento \\ Prov. Lecce-Arnesano \\
73100,  Lecce \\ Italy.}
\email{giovanni.calvaruso@unisalento.it; perroneantonella@outlook.it}


\date{}

    \renewcommand{\theequation}{\thesection .\arabic{equation}}

\begin{abstract}
We completely characterize {cosymplectic  and $\alpha$-cosymplectic} Lie algebras in terms of corresponding symplectic Lie algebras and suitable derivations on them. Several examples are given and classification results are obtained in dimension five for cosymplectic, $K$-cosymplectic and coK\"ahler Lie algebras.
\end{abstract}

\maketitle

\section{Introduction}


A {\em cosymplectic manifold}, as originally defined by Libermann \cite{L1} on 1958, is a smooth manifold, of dimension $2n + 1$, admitting a closed $1$-form $\eta$ and a closed $2$-form $\omega$, such that 
$\eta \wedge \omega^n$ is a volume form. Different but related definitions of  \lq\lq cosymplectic manifolds\rq\rq, involving an almost contact metric structure,  were given in \cite{B2}
and \cite{Li}, and  we may refer to Remark~\ref{remcos} for an illustration of these distinct definitions and their interplay.

When the cosymplectic structure $(\eta,\omega=\Phi)$ arises from the $1$-form $\eta$ and the fundamental $2$-form $\Phi$ of an almost contact metric manifold $(M,\varphi,\xi,\eta,g)$, we have an 
{\em almost coK\"ahler  manifold}. Among almost coK\"ahler  manifolds,  subclasses of special interest are the ones of {\em $K$-cosymplectic manifolds}, defined by the property that the Reeb vector field $\xi$ is Killing, and {\em coK\"ahler  manifolds}, that is, normal almost coK\"ahler manifolds. 
Cosymplectic manifolds and their above cited subclasses have been recently investigated by several authors. Some examples of these studies may be found in  \cite{BG}, \cite{Da}, \cite{FV},  \cite{Li} and references therein, and we may refer to the recent survey \cite{CNY} for a thorough presentation of cosymplectic manifolds and related topics.

{\em Locally conformally cosymplectic structures} are almost cosymplectic structures, which  lie  in the conformal class of a cosymplectic structure. Several results have been obtained for locally conformally cosymplectic structures arising from an almost contact metric structure, and in particular for the related notion of {\em almost $\alpha$-Kenmotsu structures} (see for example \cite{Dileo} and references therein). At the same time, up to our knowledge, a systematic approach to the metric-less case has not been done yet.  



Following the same scheme of the cosymplectic case, by an {\em $\alpha$-cosymplectic structure} we mean an almost cosymplectic structure $(\eta,\omega)$, formed by a closed $1$-form $\eta$ and a $2$-form $\omega$ satisfying $d\omega=2\alpha \eta \wedge \omega$ for some real constant $\alpha$. They are a natural object to study. In fact, we shall see that left-invariant $\alpha$-cosymplectic structures on Lie groups are all and the ones locally conformally cosymplectic structures  with a closed $1$-form.

The aim of this paper is to describe the general structure of cosymplectic Lie algebras (equivalently, of left-invariant cosymplectic structures on Lie groups) {and of $\alpha$-cosymplectic Lie algebras, and related subclasses}.  We begin in Section~2 introducing  some basic definitions and results  on this topic. In Section~3 we shall discuss the relationships between (almost) cosymplectic, $K$-cosymplectic and (almost) coK\"ahler structures on odd-dimensional Lie algebras, and their generalizations given by  $\alpha$-cosymplectic and (almost)  $\alpha$-coK\"ahler Lie algebras, and corresponding properties on their even-dimensional subalgebras determined by the kernel of the $1$-form $\eta$,  also providing several explicit examples.  Isomorphisms between  such different types of  Lie algebras will be characterized in Section~4. In Section~5 we then apply {these results  to classify some} remarkable classes of cosymplectic Lie algebras in dimension five.

\bigskip\noindent
\section{Preliminaries}
\setcounter{equation}{0}

\subsection{Cosymplectic structures}
We report below some basic information on cosymplectic structures, {referring} to the recent survey \cite{CNY} for more details. 

\begin{definition}\label{sympl} {\em An {\em almost cosymplectic structure} on a smooth manifold $M$ of odd dimension $2n+1$ is a pair
$(\eta, \omega)$, where $\eta$ is a $1$-form and $\omega$ is a $2$-form, such that $\eta \wedge \omega ^n$ is a volume form on $M$. When $d\eta=d\omega=0$, the structure is said to be {\em cosymplectic}.
}\end{definition}

Any almost cosymplectic structure $(\eta,\omega)$ uniquely determines a smooth vector field $\xi$ on $M$, called the {\em Reeb vector field} of the almost cosymplectic manifold $(M, \eta, \omega)$, completely characterized by conditions
$$\eta(\xi)=1, \qquad i_\xi \omega=0 \;\, \text{(that is}, \; \xi \in \ker \omega \text{).}$$

Almost coK\"ahler and coK\"ahler structures are special cases of cosymplectic structures, where tensors $\eta$ and $\omega$ are related to the existence of an almost contact metric structure (see Definition \ref{coK} and Remark \ref{remcos}). An {\em almost contact structure}   $(\varphi,\xi,\eta)$ on a $(2n+1)$-dimensional smooth manifold $M$ is formed by a $(1,1)$-tensor $\varphi$, a global vector field $\xi$ and a $1$-form $\eta$, satisfying
 \begin{equation}\label{almostc}
 \eta (\xi)=1 \quad  {\rm and} \quad \varphi ^2 = -Id +\eta \otimes \xi
\end{equation}
(which necessarily yield $\varphi (\xi)=0$ and $\eta \circ \varphi =0$). A Riemannian metric $g$ on $M$ is said to be {\em compatible} with the almost contact structure $(\varphi,\xi,\eta)$ when
\begin{equation}\label{compg}
g(\varphi X, \varphi Y) =g(X,Y) -\eta(X)\eta(Y).
\end{equation}
In this case, $(\varphi,\xi,\eta,g)$ is said to be an {\em almost contact metric structure}.
 The {\em fundamental $2$-form} $\Phi$ of an almost  contact metric structure $(\varphi,\xi,\eta,g)$ is defined by $\Phi(X,Y)=g(X,\varphi Y)$, for all tangent vector fields $X,Y$. 

It is a remarkable fact that for any almost contact metric structure $(\varphi,\xi,\eta,g)$, one has $\eta \wedge \Phi^n \neq 0$ (which leads to the orientability of any almost contact metric manifold \cite[p. 23]{B}). Thus, in particular {\em any almost contact metric structure $(\varphi,\xi,\eta,g)$ yields an almost cosymplectic structure}, given by $(\eta,\Phi)$.  A sort of converse result also holds, in the sense that for any almost cosymplectic manifold $(M,\eta,\omega)$, one can construct by polarization an almost contact metric structure $(\varphi, \xi,\eta,g)$ on $M$, with the same $\eta$ and $\xi$, whose fundamental $2$-form is given by $\omega$  {(see Theorem 3.3 in \cite{CNY}) }.

An almost contact metric structure $(\varphi,\xi,\eta,g)$ is said to be {\em normal} if the almost complex structure $J$ on $M \times \mathbb R$, defined by
$$J\left(X,f\frac{d}{dt} \right)=\left(\varphi X-f \xi, \eta (X) \frac{d}{dt} \right),$$ 
is integrable. It is well known \cite{B} that the integrability of the above $J$ is equivalent to
\begin{equation}\label{normal}
N_{\varphi}+2d\eta \otimes \xi=0,
\end{equation}
where $N_{\varphi}$ denotes the {\em Nijenhuis tensor} of $\varphi$. Moreover, we note that if an almost contact metric structure $(\varphi,\xi,\eta,g)$  satisfies $d\eta=0$, the normality condition \eqref{normal} reduces to $N_\varphi=0$. We recall the following.

\begin{definition}\label{coK} {\em An {\em almost coK\"ahler manifold} is an almost contact metric manifold $(M, \varphi, \xi, \eta, g)$, such that $(\eta,\Phi)$ is a cosymplectic structure. If in addition the almost contact structure is normal, then $M$ is called a {\em coK\"ahler manifold}.
}
\end{definition}

\begin{remark} \label{remcos}{\em 
Some explanations are required here, since different authors used different names for these structures. CoK\"ahler manifolds were introduced and first studied by Blair \cite{B2} under the name of \lq\lq cosymplectic manifolds\rq\rq.  Since then, several authors used this name for coK\"ahler manifolds (see for example \cite{Da},\cite{FV}), although cosymplectic structures, in the wider but related sense reported in the above  Definition~\ref{sympl}, had been previously introduced by Libermann \cite{L1,L2}. Moreover, since the work \cite{Li}, in several papers \lq\lq  cosymplectic manifolds\rq\rq \ {denote} {\em almost} coK\"ahler manifolds as defined in  Definition~\ref{coK}. 

In this paper we  are following  the notations used in \cite{CNY}, which seem well suited to avoid misunderstanding among all these different but related classes.
}\end{remark}

Just like in the contact metric case, an important intermediate class arises between the ones of almost coK\"ahler and coK\"ahler structures,  by requiring that the Reeb vector field is Killing \cite{BG}.

\begin{definition}\label{Kcos} {\em An almost coK\"ahler manifold $(M,\varphi,\xi,\eta,g)$ is said to be {\em $K$-cosymplec\-tic} if its Reeb vector field $\xi$ is Killing.
}
\end{definition}

\noindent
It is easily seen that any coK\"ahler manifold is $K$-cosymplec\-tic, while the converse does not hold \cite{BG}. Moreover, Proposition 2.2 and Corollary 2.4 in \cite{BG} yield the following characterization.

\begin{proposition}\label{lemmaBG}{\bf \cite{BG}}
For an almost coK\"ahler manifold $(M,\varphi,\xi,\eta,g)$, the following properties are equivalent:
\begin{itemize}
\vspace{4pt}\item[(i)] $(M,\varphi,\xi,\eta,g)$ is $K$-cosymplectic (that is, $\xi$ is Killing);
\vspace{4pt}\item[(ii)] $N^{(3)}=\mathcal{L} _{\xi}\varphi$$=0$, where $\mathcal{L}$ denotes the Lie derivative;
\vspace{4pt}\item[(iii)] $\xi$ is parallel (that is, $\nabla \xi=0$, where $\nabla$ is the Levi-Civita connection of $g$).
\end{itemize}
\end{proposition}{}

\subsection{$\alpha$-cosymplectic structures.}


An almost cosymplectic structure $(\eta,\omega)$ on an odd-dimensional manifold $M$ is said to be {\em locally conformally cosymplectic} if  for any point $p \in M$ there exists an open neighborhood $U$ of $p$ and a smooth function $\sigma:U\to \mathbb R$, such that
$$d(e^{-\sigma}\eta|_U)=0 \qquad \text{and} \qquad  d(e^{-2\sigma}\omega|_U)=0.$$
If the above equations hold globally on $M$, then  $(\eta,\omega)$ is said to be {\em globally conformally cosymplectic}.

Following the scheme of the locally conformally symplectic case (see \cite{Lee}, \cite{Vaisman1}), we obtain the following characterization.

\begin{theorem}\label{lccchar}
An almost cosymplectic structure $(\eta,\omega)$ is locally conformally cosymplectic if and only if there exists a globally defined closed $1$-form $\tilde\eta$, such that 
\begin{equation}\label{lcs}
d\eta=\tilde\eta \wedge \eta \qquad \text{and} \qquad d\omega=2\tilde\eta \wedge \omega,
\end{equation}
and $(\eta,\omega)$ is globally conformally cosymplectic if and only if $\tilde\eta$ is exact.

In particular, $(\eta,\omega)$ is locally conformally cosymplectic  with $\eta$ closed if and only if there exists some smooth function $f: M \to \mathbb R$, such that $d\omega=2 f \eta \wedge \omega$.
\end{theorem}

\begin{proof}
If $(\eta,\omega)$ is locally conformally cosymplectic, then on each neighborhood $U$, from $d(e^{-\sigma}\eta|_U)=0$ and $d(e^{-2\sigma}\omega|_U)=0$  we have at once { $d\eta=\tilde\eta{_U} \wedge \eta$ and  $d\omega=2\tilde\eta{_U} \wedge \omega$, where we put $\tilde\eta{_U}=d\sigma$. Denote by $U$ and $V$, $U\cap V\neq0$, two neighborhoods which satisfy the above conditions. Since $(\eta,\omega)$ is almost cosymplectic, $\omega$ is nondegenerate. Therefore, as $(\tilde\eta_U -\tilde\eta_V )\wedge \omega=0$ we conclude that $\tilde\eta_U=\tilde\eta_V$  on $U\cap V$. Hence, we have a global $1$-form $\tilde\eta$ on $M$ satisfying \eqref{lcs}, uniquely determined by $\tilde\eta|_U=\tilde\eta_U$} (which also ensures that $\tilde\eta$ is closed).

Conversely, since $d\tilde\eta=0$, for any open neighborhood $U$ on $M$ where $\tilde\eta$ is exact, that is, $\tilde\eta|_U=d\sigma$, from \eqref{lcs} we have at once that $d(e^{-\sigma}\eta|_U)=0$ and $d(e^{-2\sigma}\omega|_U)=0$.

{With regard to the last statement of Theorem~\ref{lccchar}, the \lq\lq if\rq\rq \ part follows at once from  the first part  by setting $\tilde\eta=f \eta$. For the \lq\lq only if\rq\rq \ part, suppose that $(\eta,\omega)$ is locally conformally cosymplectic and $d\eta=0$. Then, again from the first part of Theorem~\ref{lccchar}, there exists some 
$1$-form $\tilde\eta$, satisfying  \eqref{lcs}. But then, $0=d\eta=\tilde\eta \wedge \eta$ and so, for any tangent vector field $X$ on $M$, we have $\tilde\eta (X)=f \eta(X)$, where we put $f=\tilde{\eta}(\xi)$.
}
\end{proof}

\noindent
We now introduce the following.

\begin{definition}
{\em An almost cosymplectic structure $(\eta,\omega)$ is said to be {\em $\alpha$-cosymplectic} if $\eta$ is closed and there exists some $\alpha \in \mathbb R$, such that $d\omega=2\alpha \eta \wedge \omega$.}
\end{definition}

Clearly, cosymplectic structures are $\alpha$-cosymplectic with $\alpha=0$. On the other hand, setting $\tilde\eta=\alpha \eta$ we see at once that equations \eqref{lcs} hold. Thus,  any $\alpha$-cosymplectic structure is locally conformally cosymplectic. {In the special case of left-invariant structures, by Theorem~\ref{lccchar} and the above Definition we have at once the following result, which  motivates the study of the class of  $\alpha$-cosymplectic structures.

\begin{corollary} A left-invariant almost cosymplectic structure $(\eta,\omega)$ with $d\eta=0$ on a Lie group $G$ is locally conformally cosymplectic if and only if it is $\alpha$-cosymplectic. 
\end{corollary}

}

\smallskip
We recall that an {\em almost $\alpha$-Kenmotsu structure} is an almost contact metric structure $(\varphi,\xi,\eta,g)$ such that $d\eta=0$ and $d\Phi= 2 \alpha \eta \wedge \Phi$ {for some non-zero real constant $\alpha$,} and an {\em $\alpha$-Kenmotsu structure} is a normal almost $\alpha$-Kenmotsu structure  {\cite{Dileo},\cite{Olszak}.} So, setting $\tilde \eta=\alpha \eta$, by Theorem \ref{lccchar} we get that $(\eta,\Phi)$ is $\alpha$-cosymplectic. Indeed, almost $\alpha$-Kenmotsu structures are locally conformal to almost coK\"ahler structures (see \cite{Vaisman2},\cite{Olszak}).

We can treat (almost) coK\"ahler and (almost) $\alpha$-Kenmotsu structures in a unified way introducing the following.

\begin{definition}
{\em An {\em almost $\alpha$-coK\"ahler structure} is an almost contact metric structure $(\varphi,\xi,\eta,g)$ such that $(\eta,\Phi)$ is $\alpha$-cosymplectic, and an {\em $\alpha$-coK\"ahler structure} is a normal almost $\alpha$-coK\"ahler structure. 
}\end{definition}


\bigskip\noindent
\section{The structure of cosymplectic and $\alpha$-cosymplectic Lie algebras}  \label{scLa}
\setcounter{equation}{0}

Let $G$ denote an odd-dimensional (simply connected) Lie group and $\mathfrak g$  the corresponding Lie algebra. Suppose that $G$ admits a left-invariant almost cosymplectic  structure $(\eta,\omega)$. Tensors $\eta$ and  $\omega$  are uniquely determined at the Lie algebra level by  a $1$-form and a $2$-form on $\mathfrak g$, respectively. Correspondingly, given a left-invariant almost contact metric structure $(\varphi,\xi,\eta,g)$, one has $\xi \in \mathfrak g$ and $\varphi$ and $g$ are respectively an endomorphism  and an inner product on $\mathfrak g$. These remarks justify the following.

\begin{definition}\label{contstr}
{\em Let $\mathfrak{g}$ denote an arbitrary $(2n+1)$-dimensional Lie algebra.
\begin{itemize}

\item[(o)]
An {\em almost contact metric structure} on $\g$ is a quadruple $(\varphi,\xi,\eta,g)$, where $\eta \in \mathfrak{g} ^*$, 
$\varphi \in {\rm End}(\mathfrak{g})$, $\xi \in \mathfrak{g}$ and $g$ is a definite positive inner product on $\g$, such that
$$\begin{array}{ll}
\eta(\xi)=1, & \quad \varphi ^2 = -I +\eta \otimes \xi, \\[4pt]
\eta\circ \varphi=0, & \quad g(\varphi X,\varphi Y)= g(X,Y)-\eta(X)\eta(Y) . 
\end{array}
$$
\item[(i)] An {\em almost cosymplectic structure} on $\g$ is a pair $(\eta,\omega)$, where $\eta \in \mathfrak{g} ^*$, $\omega \in \Lambda^2(\g)$ and $\eta \wedge {\omega ^n} \neq 0$.

%

\item[(ii)] An {\em $\alpha$-cosymplectic structure} on $\g$ is an almost cosymplectic structure $(\eta,\omega)$, such that $d\eta=0$ and 
$d\omega=2 \alpha \eta \wedge \omega$, for some real constant $\alpha$.

\item[(iii)] An {\em almost $\alpha$-coK\"ahler structure} on $\g$ is an almost contact metric structure \linebreak $(\varphi,\xi,\eta,g)$, such that $(\eta,\Phi)$ is $\alpha$-cosymplectic, where $\Phi$ denotes its  fundamental $2$-form.
%
%
\item[(iv)] {An} {\em $\alpha$-coK\"ahler structure} on $\g$ is a normal almost $\alpha$-coK\"ahler structure.
\end{itemize}
%

\noindent
An {\em almost cosymplectic} (respectively, {\em $\alpha$-cosymplectic}, {\em (almost) $\alpha$-coK\"ahler})  {\em Lie algebra} is a Lie algebra $\g$ admitting an {\em almost cosymplectic} (respectively, $\alpha$-cosymplectic, (almost) $\alpha$-coK\"ahler) structure.

}\end{definition}

%
%
%
%
%
%

Clearly,  the above  definitions give  cosymplectic and (almost) coK\"ahler Lie algebras setting $\alpha=0$, and (almost) $\alpha$-Kenmotsu structures for $\alpha\neq 0$.

Correspondingly, for an even-dimensional Lie algebra $\h$ we  have the following.

\begin{definition}\label{symplstr}
{\em Let $\mathfrak{h}$ denote an arbitrary $(2n)$-dimensional Lie algebra.

\begin{itemize}
\item[(i)] An {\em almost symplectic structure} on $\h$ is a $2$-form $\Omega$ on $\h$ of maximal rank, that is, such that $\Omega^n \neq 0$. 

\item[(ii)] A {\em symplectic structure} on $\h$ is a closed $2$-form $\Omega$ on $\h$ of maximal rank.
\item[(iii)] An {\em almost K\"ahler structure} on $\h$ is a pair $(J,h)$, where $J$ is an almost complex structure and 
$h$ is a positive definite inner product on $\h$, such that its fundamental $2$-form $\Omega$, defined by
$$
\Omega (X,Y)=h(X,JY),
$$
is  a symplectic structure. 
\item[(iv)] A {\em K\"ahler structure} on $\h$ is an almost K\"ahler structure $(J,h)$, satisfying the additional integrability condition $N_{J}=0$.
\end{itemize}
An {\em almost symplectic} (respectively, {\em symplectic, almost K\"ahler, K\"ahler}) {\em Lie algebra} is a Lie algebra $\h$ admitting an almost symplectic  (respectively, symplectic, almost K\"ahler, K\"ahler) structure. 
}\end{definition}

The results of this section will show that the correspondence between the above structures on odd- and even-dimensional Lie algebras is not just formal but substantial.  We start with the following.

\begin{theorem}  \label{teoalmostcosy}
There exists a one-to-one correspondence between $(2n+1)$-dimensional almost cosymplectic Lie algebras $(\g, \eta, \omega)$ with $d\eta=0$, and $(2n)$-dimensional almost symplectic Lie algebras $(\h, \Omega)$, together with any derivation $D\in Der(\h)$. 
\end{theorem}

\begin{proof}
We first start from an almost cosymplectic Lie algebra $(\g,\eta, \omega)$ with $d\eta=0$. Thus, $\eta \wedge \omega^n \neq 0$ and there exists a unique vector field $\xi$ completely determined by $\eta(\xi)=1$ and $i_\xi \omega=0$.

We now put $\h:= \ker \eta$ and $\Omega:=\omega|_{\h \times \h}$. We have to show that $\h$ is a Lie algebra, $\Omega$ is an almost symplectic structure on $\h$ and $(\g, \eta, \omega)$ induces a derivation $D \in Der(\h)$.

\noindent
For all $x,y\in\h$, we have 
$${\eta([x,y])}=-d\eta(x,y)=0.$$
Therefore, $[x,y]\in \ker\eta=\h$ and so, $\h$ is a Lie algebra. Consider now a basis $\{\xi,x_1,...,x_{2n}\}$ of $\g$, with $x_1,...,x_{2n} \in \ker \eta$. Then,
$$ 0\neq (\eta\wedge \omega^n)(\xi,x_1,...,x_{2n})=\eta(\xi)\Omega(x_1,x_2)...\Omega(x_{2n-1},x_{2n})=\Omega^n(x_1,...,x_{2n}).
$$
So, $(\h,\Omega)$ is an almost symplectic Lie algebra.
Next, it is well known that $\ad_\xi: \g \to \g$ is a derivation of the Lie algebra $\g$. We observe that for any $x \in \h$, we have
$${\eta([\xi,x])}=-d\eta(\xi,x)=0$$
and so, $D:=\ad _{\xi}|_{\h}$ is a derivation of $\h$. 

\medskip
Conversely, consider a $(2n)$-dimensional almost symplectic Lie algebra $(\h,\Omega)$ and a derivation $D$ of $\h$. Then, we will show that there exists a corresponding $(2n+1)$-dimensional almost cosymplectic Lie algebra $(\g,\eta,\omega)$ with $d\eta=0$. We first define a Lie algebra $\g:= \mathbb R \xi \oplus \h$, putting 
\begin{align} \label{parentesi}
[x,y]:=[x,y]_{\h}, \quad [\xi,x]:=Dx  \qquad {\rm {for\ all\ }} x,y \in \h.	
\end{align}
On $\g$, we then define a $2$-form $\omega$ setting
\begin{align*}
	\omega(\xi, x):=0, \quad \omega(x,y):=\Omega(x,y)  \qquad {\rm {for\ all\ }} x,y \in \h.
\end{align*}
Next, we  introduce $\eta: \g \to \mathbb R$, uniquely determined by conditions $\eta(\xi):=1$ and $\eta|_{\h}:=0$. Taking into account Equation~\eqref{parentesi}, we have
\begin{align*}
{d\eta(x,y)=-\eta ([x,y])=0, \quad \quad d\eta(\xi,x)=-\eta([\xi,x])=0,} \qquad \text{for all} \; x,y \in \h.
\end{align*}
So, $d\eta=0$ on $\g$. Besides, consider a basis $\{\xi,x_1,...,x_{2n}\}$ of $\g$, with $\{x_1,...,x_{2n}\}$ a basis of $\h$. Since $\Omega$ is a volume form on $\h$, we then have 
$$(\eta \wedge \omega^n) (\xi,x_1,\dots,x_{2n})=\Omega (x_1,\dots,x_{2n}) \neq 0,$$
that is, $(\g,\eta,\omega)$ is an almost cosymplectic Lie algebra and this concludes the proof.

\end{proof}

The above Theorem~\ref{teoalmostcosy} yields that the class of almost cosymplectic Lie algebras $(\g,\eta,\omega)$ with a closed $1$-form $\eta$ is indeed large. In fact, one can construct one of such examples starting from any almost symplectic Lie algebra and any derivation on it. 

\begin{example}{\bf Almost cosymplectic structures from $4D$ unsolvable Lie  algebras.}
{\em  It is well known that a four-dimensional symplectic Lie algebra is necessarily solvable \cite{chu}. Therefore, no symplectic structures exist on four-dimensional unsolvable Lie algebras $\h=\mathbb R e_0 \oplus \mathfrak{g}_3$, where $\g _3 =span(e_1,e_2,e_3)$ is either $\mathfrak{sl}(2)$ or  $\mathfrak{su}(2)$, that is,
$$[e_1,e_2]=\lambda_3 e_3, \ [e_2,e_3]=\lambda_1 e_1, \ [e_3,e_1]= \lambda_2 e_2, \ [e_0,e_i]=0, \ \  i=1,2,3,$$
with {$ \lambda_1 \lambda_2 \lambda_3 \neq 0$.} Standard calculations yield that derivations on such a Lie algebra $\h$ take the form
$$
D=\left(
\begin{array}{cccc}
  a &  0   &   0    &  0  \\[4pt]
	0 &  0   &   b    &  c  \\[4pt]
	0 &  -\frac{\lambda_2}{\lambda_1} b   &   0    &  d  \\[4pt]
	0 &  -\frac{\lambda_3}{\lambda_1} c   &  -\frac{\lambda_3}{\lambda_2} d    &  0  \\[4pt]
	\end{array}
\right), \quad a,b,c,d \in \mathbb R.
$$
By Theorem~\ref{teoalmostcosy}, starting from any nondegenerate $2$-form $\Omega$  and any of such derivations $D$ on $\h$, one can construct a corresponding almost cosymplectic structure $(\eta,\omega)$ on $\g=\mathbb R \oplus \h$, with $d\eta=0$. 
 
}
\end{example}

\begin{example}{\bf Almost cosymplectic structures from  $\h_{2n+1}\times \mathbb R$.}
{\em 
Consider  the $(2n+1)$-dimensional Heisenberg Lie algebra  $\h_{2n+1}=$ span$\{e_1,\dots,e_{2n+1}\}$, completely described by 
$$[e_i,e_{i+1}]=e_{2n+1}, \qquad i=1,\dots,n-1.$$
A linear map $\bar D:\h_{2n+1} \to \h_{2n+1}$ is a derivation of $\h_{2n+1}$ if and only if with respect to the basis $\{e_i\}$, it is described by matrix
\begin{align*}
	\bar D=\left( 
	\begin{array} {cc} 
	      			A & * \\[4pt]
							0 & \lambda
			\end{array} \right),  \quad
			&  A=(a_{i,j}) \in \mathfrak{gl_{2n}}(\mathbb R), a_{i+1,i+1}=\lambda-a_{i,i},\ i=2k+1,\ k=0,...,n-1
\end{align*}
(see \cite[Lemma 4.1]{Ovando1}). Then, $\lambda={\rm tr}\bar D/n$. 

Consider the semidirect product extension $\h:=\mathbb R e_0 \ltimes \h_{2n+1}$ defined by $[e_0,e_i]=\bar D(e_i)$ for all indices $i$.  
This extension is unimodular if and only if $\lambda=0$. As proved in \cite[Proposition~4.2]{Ovando1}, such an unimodular extension $\h$ does not admit any symplectic structure. On the other hand, it obviously admits plenty of almost symplectic structures, since it suffices to consider any nondegenerate $2$-form $\Omega$ on $\h=\mathbb R e_0\times \h_{2n+1}$ (for example, $\Omega:=e^{01}+{e^{23}+ \dots} +e^{2n 2n+1}$, where $e^{ij}:=e^i\wedge e^j$). 

As  in the  previous example, by Theorem~\ref{teoalmostcosy}, starting from any nondegenerate $2$-form $\Omega$ and any derivation $D$ on $\h$, one can construct a corresponding almost cosymplectic structure $(\eta,\omega)$ on $\g=\mathbb R \xi \oplus \h$, with $d\eta=0$.
We observe that the class of these structures is very large, because the examples of unimodular semidirect extensions $\h:=\mathbb R e_0 \ltimes \h_{2n+1}$, the nondegenerate $2$-forms $\Omega$ on them, and the derivations on them, all increase quadratically in function of $n$.
}\end{example}

We now turn our attention to $\alpha$-cosymplectic  and symplectic Lie algebras.  
Let $(\h,\Omega)$ denote a symplectic Lie algebra. An element $\theta \in \mathfrak{gl}(\h)$ is called an {\em infinitesimal symplectic transformation} (for short, an {\em i.s.t.}) if 
$$\theta^T \circ \Omega +\Omega \circ \theta=0$$
(see also \cite{BFM}). Observe that the above condition is equivalent to the fact that the bilinear operator 
$$
\begin{array}{rcl}  F_{\Omega, \theta}:\quad \h \times \h &\to& \mathbb R \\[4pt] (x,y) &\mapsto& \Omega(\theta(x),y) 
\end{array}$$
is symmetric. We now prove the following result.

\begin{theorem}  \label{teocosy}
There exists a one-to-one correspondence between $(2n+1)$-dimensional $\alpha$-cosymplectic Lie algebras $(\g,\eta,\omega)$, and $(2n)$-dimensional symplectic Lie algebras $(\h, \Omega)$ together with a derivation $D\in Der(\h)$, such that $D+\alpha I$  is an infinitesimal symplectic  transformation.
\end{theorem}

{\begin{proof}
We first start from an $\alpha$-cosymplectic Lie algebra $(\g,\eta, \omega)$. By the proof of Theorem~\ref{teoalmostcosy}, we know that  $\h:= \ker \eta$ is a Lie algebra and $\Omega:=\omega|_{\h \times \h}$ an almost symplectic structure on it. Moreover, from $d\omega=2\alpha \eta \wedge \omega$ it follows at once $d\Omega=0$ and so, $(\h,\Omega)$ is a symplectic Lie algebra.

\noindent
Next, again by the proof of Theorem~\ref{teoalmostcosy}, $D:=\ad _{\xi}|_{\h}$ is a derivation of $\h$. Finally, for all $x,y \in \h$ we find
\begin{align*}
2\alpha \Omega (x,y)&=d\omega(\xi,x,y)=-\omega([\xi,x],y)-\omega([x,y],\xi)-\omega([y,\xi],x)\\[4pt]
& = \omega([\xi,y],x)-\omega([\xi,x],y)=\Omega(D y,x)-\Omega(D x,y),
\end{align*}
that is, 
$$\Omega(D x,y)+\alpha \Omega(x,y)=\Omega(D y,x)+\alpha \Omega (y,x)$$
and so, $D+\alpha I$ is an infinitesimal symplectic transformation.

\medskip
Conversely, consider a $(2n)$-dimensional symplectic Lie algebra $(\h,\Omega)$ and suppose that there exists a derivation $D$ of $\h$ such that $D+\alpha I$ is an infinitesimal symplectic transformation. By the proof of Theorem~\ref{teoalmostcosy},  we know that \eqref{parentesi} defines a Lie algebra structure on $\g:= \mathbb R \xi \oplus \h$, admitting an almost cosymplectic structure $(\eta,\omega)$, defined by
$$\begin{array}{lll} 
&\begin{array}{lll} 
   &\omega(\xi, x)=0, &\quad \omega(x,y)=\Omega(x,y),  \\[2pt]
   &\eta(\xi)=1, &\quad  \eta(x)=0,
\end{array}
 &\quad \text{for all} \; x,y \in \h
\end{array}$$
and satisfying $d\eta=0$. We now check that  $d\omega=2\alpha \eta \wedge \omega$. In fact, for any $x,y,z \in \h$ we have at once $d\omega(x,y,z)=d\Omega(x,y,z)=0$. Moreover, $d\omega$ being a $3$-form, we have  $d\omega(\xi,\xi,x)=0$ for all $x\in \h$. Finally, using the fact that $D+\alpha I$ is an infinitesimal symplectic transformation, for all $x,y\in\h$ we find
\begin{align*}
	d\omega(\xi,x,y)&=-\omega([\xi,x],y)-\omega([x,y],\xi)-\omega([y,\xi],x)=-\omega(Dx,y)+\omega(Dy,x)\\[4pt]
	& = -\Omega(Dx,y)+\Omega(Dy,x)= \Omega (\alpha x,y)- \Omega (\alpha y,x)=2\alpha \Omega (x,y).
\end{align*}
Hence, $(\g,\eta,\omega)$ is an $\alpha$-cosymplectic Lie algebra and this concludes the proof.
\end{proof}
}

\noindent
In the special case $\alpha=0$, the above Theorem~\ref{teocosy} leads to the characterization of cosymplectic Lie algebras (see {also} \cite{BFM}).

\begin{remark}
{\em Consider a symplectic Lie algebra $(\h,\Omega)$, a derivation $D$ on $\h$ and the Lie algebra extension $\g:=\mathbb R \xi \ltimes \h$, defined by $[\xi ,x]=D(x)$ for all $x\in \h$. Extend $\Omega$ to a $2$-form $\omega$ on $\g$ setting $\omega(\xi,\cdot)=0$. Then, for all $x,y \in 	\h$, by direct calculation we have $(\mathcal{L}_{\xi} \omega )(\xi,x)=0=2\alpha \omega (\xi,x)$ and  
$$(\mathcal{L}_{\xi} \omega )(x,y)= -\Omega([\xi,x],y)-\Omega(x,[\xi,y])= -\Omega(D(x),y)-\Omega(x,D(y)),$$
so that $D+\alpha I$ is an infinitesimal symplectic transformation if and only if $\mathcal{L}_{\xi} \omega =2\alpha \omega$.
}\end{remark}

\begin{example}
{\em 
By the above Theorem~\ref{teocosy}, in order to find examples of $(2n+1)$-dimensional $\alpha$-cosymplectic (in particular, cosymplectic) Lie algebras, it suffices to start with a $(2n)$-dimensional symplectic Lie algebra $(\h,\Omega)$ and to find  derivations $D$ over $\h$, such that  $D+\alpha I$ (in particular, $D$) is an infinitesimal symplectic transformation. 


We explicitly observe that such a derivation always exists for $\alpha=0$: {one can then} consider $D=0$, for which  $F_{\Omega,D}=0$ is trivially symmetric.

A classification (up to isomorphisms) of nilpotent symplectic Lie algebras of dimension $\leq 6$ was given in \cite{KGM}. In particular, starting from each of the $24$ classes of $6$-dimensional symplectic Lie algebras described in \cite{KGM}, one can find  all the corresponding $7$-dimensional cosymplectic and $\alpha$-cosymplectic Lie algebras.

For example, let us consider the following $6$-dimensional Lie algebra:
$$\h : \quad [X_1,X_2] = X_3, \quad [X_1,X_3] = X_4, \quad [X_1,X_4] = X_5, \quad [X_1,X_5] = X_6.$$
By \cite{KGM}, the above Lie algebra $\h$ carries the symplectic form 
$$\Omega= X^{16}-X^{25}+X^{34},$$ 
unique up to isomorphisms, where $X^{ij}=X^i \wedge X^j$, for all indices $i,j$. 

Long but standard calculations yield that $D \in Der(\h)$ if and only if, with respect to the basis $\{X_i \}$, $D$ is described by the matrix
$$D=\left(\begin{array}{cccccc}
d_{1,1} & 0 & 0& 0& 0& 0 \\[2pt]
d_{2,1} & d_{2,2} & 0& 0& 0& 0 \\[2pt]
d_{3,1} & d_{3,2} & d_{1,1}+d_{2,2}& 0& 0& 0 \\[2pt]
d_{4,1} & d_{4,2} & d_{3,2}& 2d_{1,1}+d_{2,2}& 0& 0 \\[2pt]
d_{5,1} & d_{5,2} & d_{4,2}& d_{3,2} & 3d_{1,1}+d_{2,2}& 0 \\[2pt]
d_{6,1} & d_{6,2} & d_{5,2}& d_{4,2}&  d_{3,2}& 4d_{1,1}+d_{2,2} \\[2pt]
\end{array}\right),$$
which depends on $11$ real parameters. We then calculate $F_{\Omega,D}$ for the arbitrary derivation described above. Requiring  that $F_{\Omega,D}(X_i,X_j)=F_{\Omega,D}(X_j,X_i)$ for all indices $i \leq j=1,\dots,6$, we get $7$ algebraic equations for the parameters describing $D$, namely,
$$d_{1,1} = d_{2,2}= d_{3,1}= d_{4,2}= 0, \quad  d_{3,2} =d_{2,1}, \quad d_{5,2}= d_{4,1},   \quad  d_{6,2} = -d_{5,1}.$$
Therefore, derivations $D$ for which $F_{\Omega,D}$ is symmetric are all the ones of the form  
$$D=\left(\begin{array}{cccccc}
0 & 0 & 0& 0& 0& 0 \\[2pt]
p & 0 & 0& 0& 0& 0 \\[2pt]
0 & p& 0 & 0& 0& 0 \\[2pt]
q & 0 & p& 0 & 0& 0 \\[2pt]
r & q & 0 & p &0 & 0 \\[2pt]
s & -r & q& 0 & p& 0 \\[2pt]
\end{array}\right),$$
\noindent
where we put $d_{2,1}=p, d_{4,1}=q, d_{5,1}=r, d_{6,1}=s$.  Consequently, applying Theorem~\ref{teocosy}, we conclude that $(\h,\Omega)$ gives rise to a family of $7$-dimensional cosymplectic Lie algebras $(\g,\eta,\omega)$, depending on $4$ parameters $p,q,r,s$ {and} explicitly described by 
$$\begin{array}{rl} \g : & [X_0,X_1]=pX_2+qX_4+rX_5+sX_6, \quad [X_0,X_2]=pX_3+q X_5-rX_6, \\[6pt]
& [X_0,X_3]=pX_4+q X_6, \quad  [X_0,X_4]=pX_5, \quad [X_0,X_5]=p X_6, \;  \\[6pt]
& [X_1,X_2] = X_3, \quad [X_1,X_3] = X_4, \quad [X_1,X_4] = X_5, \quad [X_1,X_5] = X_6,
\end{array}$$
$$\eta=X^0, \qquad \omega= X^{16}-X^{25}+X^{34}.$$

More in general, if we now require  that $F_{\Omega,D+\alpha I}$ is symmetric, then we get all and the ones derivations $D$  of the form  
$$D=\left(\begin{array}{cccccc}
-\frac{2}{7}\alpha & 0 & 0& 0& 0& 0 \\[2pt]
p & -\frac{4}{7}\alpha & 0& 0& 0& 0 \\[2pt]
0 & p& -\frac{6}{7}\alpha & 0& 0& 0 \\[2pt]
q & 0 & p& -\frac{8}{7}\alpha & 0& 0 \\[2pt]
r & q & 0 & p &-\frac{10}{7}\alpha & 0 \\[2pt]
s & -r & q& 0 & p& -\frac{12}{7}\alpha \\[2pt]
\end{array}\right),$$
\noindent
and by Theorem~\ref{teocosy}, each of them  gives rise to a corresponding $\alpha$-cosymplectic Lie algebra. 
}\end{example}

\bigskip

We now consider almost $\alpha$-coK\"ahler Lie algebras and prove the following.

\begin{theorem}  \label{teoalmostcokah}
There exists a one-to-one correspondence between $(2n+1)$-dimensional almost {$\alpha$-coK\"ahler} Lie algebras $(\g,\varphi,\xi,\eta, g)$, and $(2n)$-dimensional almost K\"ahler Lie algebras $(\h, J, h)$, admitting a derivation $D\in Der(\h)$ such that $D+\alpha I$ is an {i.s.t.}
%
\end{theorem}

\begin{proof}
Let $(\g,\varphi,\xi,\eta, g)$ be a $(2n+1)$-dimensional almost $\alpha$-coK\"ahler  Lie algebra. In particular, $(\g,\eta,\Phi)$ is then {$\alpha$-cosymplectic} and so, by the proof of Theorem~\ref{teocosy}, $(\h:=\ker \eta, \Phi|_{\h \times \h})$ is a symplectic Lie algebra. On $\h$, we define
$$J:=\varphi|_{\h}\qquad  \text{and} \qquad h:=g|_{\h\times \h} .$$
By the properties of $(\varphi,\xi,\eta, g)$ it then follows at once that $h$ is a positive definite inner product on 
$\h$, $J$ a complex structure on $\h$ such that $h(J\cdot,J\cdot)=h(\cdot,\cdot)$, and the fundamental $2$-form of $(J,h)$ is
$$\Omega(x,y):=h(x,Jy)=g(x,\varphi y)=\Phi(x,y) \quad \text{for all} \; x,y \in \h,$$
that is, $\Omega=\Phi|_{\h \times \h}$. Since, by hypothesis, {$d\Phi=2\alpha \eta \wedge \Phi$} and $\eta\wedge \Phi^n\neq 0$, one gets at once that $\Omega$ is a closed nondegenerate $2$-form on $\h$. Consequently, $(\h,J,h)$ is an almost K\"ahler Lie algebra. Finally, again from the proof of Theorem~\ref{teocosy}, we have that $D:=\ad_\xi|_{\h}$ is a derivation of $\h$, such that {$D+\alpha I$ is an infinitesimal symplectic transformation.}

Conversely, proceeding in a similar way to the proof of Theorem~\ref{teocosy}, starting from an almost K\"ahler Lie algebra $(\h, J, h)$, together with a derivation $D\in Der(\h)$ such that {$D+\alpha I$ is an infinitesimal symplectic transformation,} we consider $(\g,\varphi,\xi,\eta,g)$ defined as follows:
\begin{align*}
& \g= \mathbb R \xi \oplus \h, {\rm \quad with \quad} [x,y]=[x,y]_\h, \ \  [\xi,x]=Dx, \quad \text{for all}\; x,y \in \h,\\[4pt] 
&\eta(\xi)=1,\ \eta|_\h:=0, \quad\quad \varphi(\xi)=0,\  \varphi|_\h=J,\quad\quad g(\xi,x)=\eta(x), \  g|_{\h\times \h}=h.
\end{align*}
\noindent
It is now easy to check that $(\varphi,\xi,\eta,g)$  is an almost contact metric structure, $d\eta=0$ and $\Phi:=g(\cdot,\varphi \cdot)$ satisfies $\Phi|_\h=\Omega$. Since $\Omega$ is a closed nondegenerate $2$-form, we have again $\eta\wedge \Phi^n \neq 0$. Moreover, taking into account  the fact  that {$D+\alpha I$ is an infinitesimal symplectic transformation, {it is easily seen that} $d\Phi=2 \alpha \eta \wedge \Phi$} and this ends the proof.
\end{proof}

By Theorem~\ref{teoalmostcokah}, the problem of determining $(2n+1)$-dimensional almost  $\alpha$-coK\"ahler Lie algebras, reduces to the one of determining $(2n)$-dimensional almost K\"ahler Lie algebras {together with} some suitable derivations. 
{Explicit examples of almost  $\alpha$-coK\"ahler Lie algebras will be given in the next section.}


\smallskip
It is well known that tensor  $\mathcal{L}_{\xi} \varphi$ plays an important role in the geometry of almost contact metric manifolds (see for example \cite{B}). We now consider  almost $\alpha$-coK\"ahler Lie algebras $(\g,\varphi,\xi,\eta, g)$ such that $\mathcal L_\xi \varphi =0$. In particular, we recall that  a {\em $K$-cosymplectic Lie algebra} is an almost coK\"ahler Lie algebra, whose Reeb vector field $\xi$ is Killing \cite{BG}. We first prove the following.

\begin{proposition} \label{propA}
Let $(\g,\varphi,\xi,\eta, g)$ denote an almost $\alpha$-coK\"ahler Lie algebra and $(\h=\ker \eta, J,h)$  the corresponding almost K\"ahler Lie algebra as in Theorem~{\em\ref{teoalmostcokah}}. Then, the following properties are equivalent:
\begin{itemize}
\vspace{4pt}\item[i)] $\mathcal L_\xi g=2\alpha(g - \eta \otimes \eta)$;
\vspace{4pt}\item[ii)] $(D+\alpha I)$ is skew-adjoint with respect to $h$, where $D:=\ad_\xi|_\h$;
\vspace{4pt}\item[iii)] $DJ=JD$;
\vspace{4pt}\item[iv)] $\mathcal L_\xi \varphi=0$.
\end{itemize}
In particular, for an almost coK\"ahler Lie algebra $(\g,\varphi,\xi,\eta, g)$ {(that is, when $\alpha=0$), by property {\em i)} we have that} $\xi$ is Killing, that is, $(\g,\varphi,\xi,\eta, g)$ is $K$-cosymplectic.
\end{proposition}

\begin{proof}
Let $(\g,\varphi,\xi,\eta, g)$ be an almost $\alpha$-coK\"ahler Lie algebra and consider $\h:=\ker\eta$, $J:=\varphi|_h$, $h:=g|_{\h \times \h}$ and $D:=\ad_\xi|_\h \in Der(\h)$. 

For any $x,y \in \g$, we get
\begin{align} \label{1}
	(\mathcal  L_\xi g)(x,y)= -g([\xi,x],y)-g([\xi,y],x)= -h(Dx,y)-h(Dy,x).
\end{align}

\noindent 
Moreover, it is well known (see for example \cite{CNY},\cite{Dileo})  that the Levi-Civita connection of $g$ satifies $ \nabla \xi= - \alpha \varphi^2- \frac{1}{2}\varphi\mathcal{L}_{\xi} \varphi$, that $\mathcal L_\xi \varphi$ is symmetric, $\varphi$ is skew-adjoint with respect to $g$ and $\mathcal{L}_{\xi} \varphi$ anti-commutes with $\varphi$. Consequently, for any $x,y \in \g$ we get
\begin{align*}
 (\mathcal L_\xi g)(x,y)&=(g(\nabla_x \xi,y) + g( x,\nabla_y \xi))\\[4pt]
	&=g(- \alpha \varphi^2 x -\frac{1}{2}\varphi\mathcal{L}_{\xi} \varphi x,y) + g(x, -\alpha \varphi^2 y- \frac{1}{2}\varphi\mathcal{L}_{\xi} \varphi y)\\[4pt]
	&= -\alpha \big( g(-x+\eta(x)\xi,y)+g(x, -y+\eta(y)\xi)\big) -  g(\varphi \mathcal{L}_{\xi} \varphi x,y)\\[4pt]
	&= 2\alpha \big( g(x,y)- \eta(x) \eta(y)\big) -  g(\varphi \mathcal{L}_{\xi} \varphi x,y).
\end{align*}
So, we have
\begin{align} \label{2}
	\mathcal L_\xi g	= 2\alpha \big( g -\eta \otimes \eta) -  g(\varphi \mathcal{L}_{\xi} \varphi\  \cdot,\cdot).
\end{align}
Equation \eqref{2} yields the equivalence between i) and iv). Moreover, comparing Equations \eqref{1} and \eqref{2}, for any $x,y\in \h$ we have
\begin{align*}
-h(Dx,y)-h(Dy,x) = 2\alpha h(x,y) - h(\varphi \mathcal{L}_{\xi} \varphi x,y).
\end{align*}
that is,
\begin{align*}
  h(\varphi \mathcal{L}_{\xi} \varphi x,y)= h((D+\alpha I)x,y) +h((D+\alpha I)y,x),
\end{align*}
{which, since $(\mathcal L_\xi \varphi) \xi=0,$ implies that} $\mathcal{L}_{\xi} \varphi=0$ if and only if $(D+\alpha I)$ is skew-adjoint with respect $h$.

Finally, iv) implies iii) because $\mathcal L_\xi \varphi=\ad_\xi \circ \varphi-\varphi \circ \ad_\xi$ and  
$J=\varphi|_{\h}$, and to prove the  converse it suffices to observe that $\varphi \xi=\ad_\xi \xi=0$. 
\end{proof}

\begin{remark}\label{p13}{\em
If an almost K\"ahler Lie algebra $(\h,J,h)$ admits a derivation $D$ such that conditions  ii) and iii) of  Proposition~\ref{propA} hold, then $D+\alpha I$ is necessarily an {i.s.t.} In fact, taking also into account that $J^2=-Id$ and that $J$ is compatible with $h$, we get, for any $x,y \in \h$, 
\begin{align*}
	F_{\Omega,D+\alpha I}(x,y)&=\Omega((D+\alpha I)x,y)=h((D+\alpha I)x,Jy)=h(J(D+\alpha I)x,J^2y) \\[4pt]
	&=-h(J(D+\alpha I)x,y)= -h((D+\alpha I)Jx,y)=h(Jx,(D+\alpha I)y)\\[4pt]
	&=\Omega((D+\alpha I)y,x)=F_{\Omega,D+\alpha I}(y,x).
\end{align*}
By  a similar argument one can prove that if an almost K\"ahler  Lie algebra $(\h,J,h)$ admits a derivation $D$ such that $D+\alpha I$  is an i.s.t. and $D$ commutes with $J$, then condition ii) of {Proposition~\ref{propA}} holds.
}\end{remark}

We now prove the following characterization.

\begin{theorem}  \label{teoKcosy}
There exists a one-to-one correspondence between  $(2n+1)$-dimensional almost $\alpha$-coK\"ahler Lie algebras $(\g,\varphi,\xi,\eta, g)$ with $\mathcal L_\xi \varphi=0$ (in particular, when $\alpha=0$, $K$-cosymplectic Lie algebras), and $(2n)$-dimensional almost K\"ahler Lie algebras $(\h, J, h)$, admitting a derivation $D\in Der(\h)$ such that {$D+\alpha I$ is an i.s.t.} and $DJ=JD$.
\end{theorem}

\begin{proof}
Let $(\g,\varphi,\xi,\eta, g)$ be a $(2n+1)$-dimensional almost $\alpha$-coK\"ahler Lie algebra with $\mathcal L_\xi \varphi=0$. Then, by Theorem~\ref{teoalmostcokah},  it determines a  $(2n)$-dimensional almost K\"ahler Lie algebra $(\h,J,h)$, admitting a derivation $D\in Der(\h)$ such that  $D+\alpha I$ is an {i.s.t.} Besides, by Proposition~\ref{propA}  we conclude that $DJ=JD$.

Conversely,  let us start from a $(2n)$-dimensional almost K\"ahler Lie {algebra} $(\h, J, h)$, admitting a derivation $D\in Der(\h)$ such that   $D+\alpha I$ is an {i.s.t.} and $DJ=JD$. By  Theorem~\ref{teoalmostcokah},  it determines a $(2n+1)$-dimensional almost $\alpha$-coK\"ahler Lie algebra $(\g,\varphi,\xi,\eta, g)$. Futhermore, by Proposition~\ref{propA}, the condition $DJ=JD$ implies that {$\mathcal L_\xi \varphi=0$ and this completes the proof.}

\end{proof}

\noindent
Applying the above Theorem~\ref{teoKcosy}, an explicit classification of four-dimensional strictly $K$-cosymplectic Lie algebras will be given in Section~5.
 
\medskip
Next, we observe that for  {an almost $\alpha$-coK\"ahler}  Lie algebra  $(\g,\varphi,\xi,\eta,g)$ with $\mathcal L_\xi \varphi=0$, the normality condition \eqref{normal}  (that is, as $d\eta=0$, $N_\varphi=0$)  is equivalent to the integrability  condition $N_J=0$ of the corresponding almost K\"ahler Lie algebra $(\h=\ker\eta,J=\varphi|_{\h},h)$. 

In fact, as $J=\varphi|_{\h}$, we get at once that $N_\varphi=0$ implies $N_J=0$. Conversely, for any $x,y \in \ker\eta$  we have $N_\varphi(x,y)=N_J(x,y)=0$.  Moreover, since $\mathcal L_\xi \varphi=0$, we  also {have}
\begin{align*}
	N_\varphi(\xi,x)&= [\varphi \xi, \varphi x] - \varphi[\varphi \xi, x] - \varphi[\xi, \varphi x] + \varphi^2[\xi,x]= \varphi(\varphi[\xi,x] -[\xi,\varphi x]) \\[4pt]
	&  =\varphi ((\mathcal L_\xi \varphi)x)=0.
\end{align*}

Thus, Theorem~\ref{teoKcosy} yields the following.

\begin{theorem}\label{teocokah}
There exists a one-to-one correspondence between $(2n+1)$-dim\-ensional $\alpha$-coK\"ahler Lie algebras $(\g,\varphi,\xi,\eta, g)$, and 
$(2n)$-dim\-en\-sional K\"ahler Lie algebras $(\h, J, h)$, admitting a derivation $D\in Der(\h)$ such that $D+\alpha I$ is an {i.s.t.}  and $DJ=JD$.
%
\end{theorem}

\noindent
The above result extends the characterization of coK\"ahler structures proved in  \cite[Theorem 6.1]{FV}.

\section{Isomorphisms of cosymplectic and $\alpha$-cosymplectic Lie algebras}
\setcounter{equation}{0}

We start with the following.

\begin{definition}\label{iso} 
{\em a) Two almost {cosymplectic (in particular, two $\alpha$-cosymplectic)} Lie \ algebras   $(\g_1,\eta_1,\omega_1)$ and $(\g_2,\eta_2,\omega_2)$ are {\em isomorphic} if there exists a Lie algebra isomorphism $\Psi: \g_1 \to \g_2$ such that

\smallskip
\centerline{ 1) $\Psi^* \omega_2=\omega_1$\quad and \quad 2) $\Psi^*\eta_2(=\eta_2 \circ \Psi)=\eta_1$.}

\medskip
b) Two (almost) $\alpha$-coK\"ahler Lie algebras $(\g_1,\varphi_1,\xi_1,\eta_1,g_1)$ and $(\g_2,\varphi_2,\xi_2,\eta_2,g_2)$ are {\em isomorphic} if there exists  a Lie algebra isomorphism $\Psi: \g_1 \to \g_2$ such that 
$$\begin{array}{ll}
 1)\  \Psi \circ \varphi_1=\varphi_2 \circ \Psi, & 2)\  \Psi^*\eta_2(=\eta_2 \circ \Psi)=\eta_1,\\[4pt]
3)\  \Psi \xi_1=\xi_2,\quad  & 4)\  \Psi^* g_2=g_1 \, (\textrm{i.e.,} \ \Psi {\textrm{\ is\  an\  isometry}}).
\end{array}$$

\medskip
c) Two (almost) symplectic Lie algebras $(\h_1,\Omega_1)$ and $(\h_2,\Omega_2)$ are {\em isomorphic} if there exists a Lie algebra isomorphism $\psi: \h_1 \to \h_2$ such that $\psi^* \Omega_2=\Omega_1$.

\medskip
d) Two (almost) K\"ahler Lie algebras $(\h_1,J_1,h_1)$ and $(\h_2,J_2,h_2)$ are {\em isomorphic} if there exists a Lie algebra isomorphism $\psi: \h_1 \to \h_2$ such that 

\smallskip
\centerline{ 1) $\psi \circ J_1=J_2 \circ \psi$ \quad and \quad 2) $\psi^* h_2=h_1$.}

\smallskip\noindent
In such a case, we observe that conditions $1)$ and $2)$ imply that $\psi^* \Omega_2=\Omega_1$, where $\Omega_1$ and $\Omega_2$ are the fundamental $2$-form of $h_1$ and $h_2$ respectively. 
}
\end{definition}

\noindent
We now prove the following.

\begin{theorem} \label{isoalmostcosy}
Let $(\g_1,\eta_1,\omega_1)$ and $(\g_2,\eta_2,\omega_2)$ two almost cosymplectic {(respectively, $\alpha$-cosymplectic)} Lie algebras with $d\eta_1=d\eta_2=0$, and $(\h_1,\Omega_1,D_1)$ and $(\h_2,\Omega_2,D_2)$ the corresponding almost symplectic Lie algebras and derivations as in Theorem~{\em\ref{teoalmostcosy}} {(respectively, Theorem~{\em\ref{teocosy}})}. Then, {$(\g_1,\eta_1,\omega_1)$ and $(\g_2,\eta_2,\omega_2)$ are isomorphic}  if and only if there exists an isomorphism $\psi: \h_1 \to \h_2$ of almost symplectic Lie algebras, such that $\psi \circ D_1= D_2 \circ \psi$.
\end{theorem}

\begin{proof}
Let $(\g_1,\eta_1,\omega_1)$ and $(\g_2,\eta_2,\omega_2)$ two almost cosymplectic Lie algebras with $d\eta_1=d\eta_2=0$, and $(\h_1,\Omega_1)$ and $(\h_2,\Omega_2)$ with $D_i\in Der(\h_i)$ ($i=1,2$) the corresponding almost symplectic Lie algebras as in Theorem~\ref{teoalmostcosy}.

Suppose first that there exists a Lie algebra isomorphism $\Psi: \g_1 \to \g_2$ satisfying $\Psi^*\omega_2=\omega_1$ and $\Psi^*\eta_2=\eta_1$.
We shall prove that $\Psi(\ker\eta_1)=\ker\eta_2$. In fact, for any $x\in\ker \eta_1$, we have 
$$0=\eta_1(x)=\eta_2 \Psi(x)$$
and so, $\Psi(\ker \eta_1)\subseteq \ker \eta_2$. But $\dim (\Psi(\ker \eta_1))=\dim (\ker \eta_2)=2n$. Hence,  $\Psi(\h_1)=\h_2$, where $\h_1=\ker\eta_1$ and $\h_2=\ker\eta_2$, and restricting $\Psi$ to $\h_1$ we get a Lie algebra isomorphism $\psi:=\Psi|_{\h_1} : \h_1 \to \h_2 $. Moreover, since $\Omega_i:=\omega_i|_{\h_1 \times \h_i}$ for $i=1,2$ we have at once that $\psi^*\Omega_2=\Omega_1$. So, $\psi$ is an isomorphism of almost symplectic Lie algebras. 

Finally, to show that $\psi \circ D_1=D_2 \circ \psi$, we first prove that $\Psi \xi_1=\xi_2$. Recall that, for $i=1,2$, $\xi_i$ is uniquely determined by  conditions  $\eta_i(\xi_i)=0$ and $\omega_i(\xi_i, \cdot)=0$. Now, put $\tilde \xi_2:=\Psi \xi_1$, and for any $y\in \h_2$, let $x\in \g_1$ such that $\Psi(x)=y$.  Then:  
\begin{align*}
	&\eta_2(\tilde \xi_2)=\eta_2(\Psi \xi_1)=\eta_1 (\xi_1)=1, \\[4pt]
	&\omega_2(\tilde\xi_2,y)=\omega_2(\Psi(\xi_1),\Psi(x))=(\Psi^* \omega_2)(\xi_1,x)=\omega_1(\xi_1,x)=0.	
\end{align*}
Consequently, $\tilde \xi_2=\xi _2$ and $\Psi(\xi_1)=\xi_2$. Therefore, for any $x\in \h_1$, we have
\begin{align*}
	\psi D_1(x)=\Psi \ad_{\xi_1}(x)=\Psi[\xi_1,x]=[\Psi \xi_1,\Psi x]=[\xi_2,\psi(x)]= \ad_{\xi_2}(\psi(x))= D_2 \psi(x).
\end{align*}

\medskip
Conversely, let $\psi: (\h_1,\Omega_1) \to (\h_2,\Omega_2)$ denote an isomorphism of almost symplectic Lie algebras, such that $\psi \circ D_1=D_2 \circ \psi$. As in the proof of Theorem~\ref{teoalmostcosy}, consider $\g_i=\mathbb R \xi_i \oplus \h_i$, for $i=1,2$. Define  a linear map $\Psi:  \g_1 \to \g_2$ by $\Psi(\xi_1)=\xi_2$ and $\Psi|{\h_1}=\psi$. Then, $\Psi$ is a Lie algebra isomorphism. In fact, 
\begin{align*}
	& \Psi(\g_1)= \Psi(\mathbb R \xi_1 \oplus \h_1)=\mathbb R \Psi(\xi_1) \oplus \Psi(\h_1)=\mathbb R \xi_2 \oplus \h_2=\g_2,   \\[4pt]
	&  \Psi[x,y]=[\Psi x, \Psi y]  \quad  \ \text{if} \quad  x,y \in \h_1 , \\[4pt]
	&  \Psi[\xi_1,x]=\Psi \ad_{\xi_1}(x)= \psi D_1(x) = D_2 \psi (x)= [\xi_2,\psi (x)]=[\Psi(\xi_1), \Psi(x)]  \quad \text{for all} \; x \in \h_1.	
\end{align*}
 
\noindent
Next, for any $x\in \h_1$,  we have  $\psi(x)\in \h_2 = \ker\eta_2$ and so,
\begin{align*}
	&\Psi^* \eta_2(x)= \eta_2(\Psi (x))= \eta_2(\psi(x))=0=\eta_1(x), \\ &	\Psi^* \eta_2(\xi_1)= \eta_2(\Psi (\xi_1))= \eta_2(\xi_2)=1=\eta_1(\xi_1).
\end{align*}
Therefore,  $\Psi^* \eta_2=\eta_1$. Finally, as we already know by hypothesis that $\psi^* \Omega_2=\Omega_1$, we  easily conclude that $\Psi^* \omega_2=\omega_1$. In fact,
\begin{align*}
	&(\Psi^* \omega_2)|_{\h_1\times \h_1}= \psi^* \Omega_2|_{\h_1\times \h_1}= \Omega_1|_{\h_1\times \h_1}=\omega_1|_{\h_1\times \h_1}, \\[4pt]
	&(\Psi^* \omega_2)(\xi_1,x)= \omega_2(\Psi (\xi_1),\Psi(x))= \omega_2(\xi_2,\psi(x))=0=\omega_1 (\xi_1,x) \quad  \text{for all} \; x \in \h_1.
\end{align*}
So,  $\Psi$ is an isomorphism of almost cosymplectic Lie algebras.

Finally, for the $\alpha$-cosymplectic case it suffices to observe that since $\Psi$ is an isomorphism of almost cosymplectic Lie algebras, $d\omega_1=2\alpha \eta_1 \wedge \omega_1$ if and only if $d\omega_2=2\alpha \eta_2 \wedge \omega_2$. Correspondingly, since $\psi$ is an isomorphism of almost  symplectic Lie algebras, one has  $d\Omega_1=0$ if and only if $d\Omega_2=0$. 
\end{proof}

%

\noindent
In the next section we shall apply the above Theorem~\ref{isoalmostcosy}, together with the classification  of non-isomorphic  four-dimensional symplectic Lie algebras obtained in \cite{Ovando1}, to get a complete classification up to isomorphisms of five-dimensional cosymplectic Lie algebras. Next, we prove the following.


\begin{theorem} \label{isoalmostcokah}
Let $(\g_1,\varphi_1,\xi_1,\eta_1,g_1)$ and $(\g_2,\varphi_2,\xi_2,\eta_2,g_2)$ two  (almost)  $\alpha$-coK\"ahler Lie algebras, and $(\h_1,J_1,h_1,D_1)$ and $(\h_2,J_2,h_2,D_2)$ the corresponding (almost) K\"ahler Lie algebras  and derivations, {as described in Theorems~{\em\ref{teoalmostcokah}} and {\em\ref{teocokah}}. Then, $(\g_1,\varphi_1,\xi_1,\eta_1,g_1)$ and $(\g_2,\varphi_2,\xi_2,\eta_2,g_2)$ are isomorphic} if  and only if there exists an isomorphism \linebreak  $\psi: \h_1 \to \h_2$ of (almost) K\"ahler Lie algebras, such that $\psi \circ D_1= D_2 \circ \psi$.
\end{theorem}

\begin{proof}
Let $(\g_1,\varphi_1,\xi_1,\eta_1,g_1)$ and $(\g_2,\varphi_2,\xi_2,\eta_2,g_2)$ denote  two {(almost)} $\alpha$-coK\"ahler Lie algebras, and $(\h_1,J_1,h_1)$ and $(\h_2,J_2,h_2)$ with $D_i \in Der(\h_i)$ ($i=1,2$) the corresponding {(almost)} K\"ahler Lie algebras  as in Theorem~\ref{teoalmostcokah}.

Suppose {first that there exists an isomorphism $\Psi: \g_1 \to \g_2$.} Following the proof of  Theorem~\ref{isoalmostcosy}, there exists a Lie algebra isomorphism  $\psi:=\Psi|_{\h_1}: \h_1  \to \h_2$,  where $\h_1:=\ker\eta_1$ and $\h_2:=\ker\eta_2$. Next, taking into account  $\Psi\circ \varphi_1=\varphi_2 \circ \Psi$ and $J_i=\varphi_i|{\h_i}$, we get
\begin{align*}
	(\psi J_1)(x)=\psi (\varphi_1(x))=\Psi (\varphi_1 (x))=\varphi_2 (\Psi(x)) = \varphi_2 (\psi(x))= (J_2 \psi)(x)   \;\, \text{for all} \; x \in \h_1,
\end{align*}
that is, $\psi \circ J_1 = J_2 \circ \psi$. Finally, since by hypothesis  $\Psi^* g_2=g_1$  and $h_i=g_i|_{\h_i\times \h_i}$ for $i=1,2$,  we find $\psi^* h_2=h_1$. Therefore, $\psi$ is an isomorphism of {(almost)} K\"ahler Lie algebras, and condition  $\psi \circ D_1= D_2 \circ \psi$ is ensured by Theorem~\ref{isoalmostcosy}.

\smallskip
Conversely, let $\psi: \h_1 \to \h_2$ denote an  isomorphism of {(almost)} K\"ahler Lie algebras, such that $\psi \circ D_1=D_2 \circ \psi$. As in Theorem~\ref{teoalmostcokah}, consider $\g_i=\mathbb R \xi_i \oplus \h_i$, for $i=1,2$. Define  a linear map  $\Psi: \g_1 \to \g_2$  by $\Psi(\xi_1)=\xi_2$ and $\Psi|_{\h_1}=\psi$. Then, as in the proof of  Theorem~\ref{isoalmostcosy}, $\Psi$ is a Lie algebras isomorphism such that $\psi^* \eta_2=\eta_1$. Moreover, $\Psi\circ \varphi_1 = \varphi_2 \circ \Psi$. In fact, for any $x \in h_1$,
\begin{align*}
	&\Psi \varphi_1(x)=\Psi (J_1 (x))=(\psi J_1)(x)=(J_2 \psi)(x)= \varphi_2 \Psi(x),\\[4pt]
	&\Psi \varphi_1(\xi_1)=\Psi (\varphi_1 \xi_1) =0= \varphi_2 \xi_2=  \varphi_2 \Psi (\xi_1).
\end{align*}
Finally, as by hypothesis $\psi$ is an isometry, for any $x,y \in \h_1$ we have
\begin{align*}
  &(\Psi^*g_2)(x,y)= g_2(\Psi x, \Psi y) = h_2(\psi x, \psi y)= (\psi^* h_2)(x,y)= h_1(x,y)=g_1(x,y), \\[4pt]
	&(\Psi^*g_2)(\xi_1,x)= g_2(\Psi \xi_1, \Psi x) = g_2(\xi_2, \psi x)=0=g_1(\xi_1,x),\\[4pt]
	&(\Psi^*g_2)(\xi_1,\xi_1)= g_2(\Psi \xi_1, \Psi \xi_1) = g_2(\xi_2,\xi_2)=1= g_1(\xi_1,\xi_1),
\end{align*}
that is, $\Psi$ is also  an isometry and so,  we conclude that $\Psi$ is an  isomorphism of {(almost)}  $\alpha$-coK\"ahler Lie algebras.

\end{proof}


The following Lemmas show that the {property of $D+\alpha I$ being an i.s.t.} and the commutativity of $D$ and $J$ are {both} stable under isomorphisms of symplectic and almost K\"ahler Lie {algebras.}

\begin{lemma} \label{lem1}
Let $\psi: (\h_1, \Omega_1) \to (\h_2, \Omega_2)$ be an isomorphism of symplectic Lie algebras and $D_i \in Der(h_i)$, $i=1,2$. If  
 $\psi \circ D_1= D_2 \circ \psi$,  then  $\psi^* F_{\Omega_2,{D_2+\alpha I}}=F_{\Omega_1,{D_1+\alpha I}}$.

\noindent
In particular, $F_{\Omega_1,{D_1+\alpha I}}$ is symmetric (that is, {$D_1+\alpha I$} is an {i.s.t.}) if and only if so is $F_{\Omega_2,{D_2+\alpha I}}$.
\end{lemma}

\begin{proof}
Let $(\h_1, \Omega_1)$ and $(\h_2, \Omega_2)$ be two symplectic Lie algebras, $D_i$ a derivation on $\h_i$, $i=1,2$, and $\psi$ an isomorphism between them. Suppose that $\psi \circ D_1= D_2 \circ \psi$. Then, taking into account $\psi^* \Omega_2 = \Omega_1$, for any $x,y \in \h_1$ we get
\begin{align*}
	(\psi^* F_{\Omega_2,{D_2+\alpha I}}) (x,y)&= F_{\Omega_2,{D_2+\alpha I}} (\psi(x), \psi(y)) = \Omega_2(({D_2+\alpha I}) \psi(x), \psi(y))\\[4pt]
	&=\Omega_2(\psi ({D_1+\alpha I})(x),\psi(y))= (\psi^* \Omega_2) (({D_1+\alpha I})(x),y)\\[4pt]
	&= \Omega_1 (({D_1+\alpha I})(x), y) = F_{\Omega_1,{D_1+\alpha I}} (x,y)
\end{align*}
{and the conclusion follows at once.
}\end{proof}

\begin{lemma}\label{lemma2}
Let $\psi: (\h_1, J_1,h_1) \to (\h_2, J_2, h_2)$ be an isomorphism of almost K\"ahler Lie algebras and $D_i \in Der(h_i)$, $i=1,2$.  If \ $\psi \circ D_1= D_2 \circ \psi$,  then $D_1J_1=J_1D_1$ if and only if $D_2J_2=J_2D_2$.
\end{lemma}

\begin{proof}
Let $\psi: (\h_1, J_1,h_1) \to (\h_2, J_2, h_2)$ be an isomorphism of almost K\"ahler Lie algebras. Then, 
$$\psi \circ J_1 = J_2 \circ \psi$$
and so, if $\psi \circ D_1= D_2 \circ \psi$, to complete the proof  it suffices to observe that
\begin{align*}
	D_1J_1 = \psi^{-1} D_2 \psi \psi^{-1} J_2 \psi = \psi^{-1} D_2 J_2 \psi  {\textrm {\; and\; }}  J_1D_1 = \psi^{-1} J_2 \psi \psi^{-1} D_2 \psi =\psi^{-1} J_2 D_2 \psi.
\end{align*}
\end{proof}

In particular,  Theorem~\ref{isoalmostcokah} and Lemma~\ref{lemma2} yield at once the following result.

\begin{theorem} \label{isoKcosy}
Let $(\g_1,\varphi_1,\xi_1,\eta_1,g_1)$ and $(\g_2,\varphi_2,\xi_2,\eta_2,g_2)$ be two almost $\alpha$-coK\"ahler Lie algebras with $\mathcal L_{\xi_i} \varphi_i =0$ (in particular, two $K$-cosymplectic Lie algebras), and $(\h_1,J_1,h_1,D_1)$ and $(\h_2,J_2,h_2,D_2)$ the corresponding almost K\"ahler Lie algebras and derivations  as in Theorem~{\em\ref{teoKcosy}}. Then,  $\g_1$ and $\g_2$ are isomorphic if  and only if there exists an isomorphism  $\psi: \h_1 \to \h_2$ of almost K\"ahler Lie algebras, such that $\psi \circ D_1= D_2 \circ \psi$.
\end{theorem}

\section{Classifications  in dimension five}
\setcounter{equation}{0}

We shall now classify up to isomorphisms five-dimensional cosymplectic (respectively,  strictly almost coK\"ahler, strictly $K$-cosymplectic, coK\"ahler) Lie algebras, {and provide some explicit examples of (almost) $\alpha$-Kenmotsu Lie algebras,} describing the form of the suitable derivations of the corresponding non-isomorphic  four-dimensional symplectic (respectively, strictly almost K\"ahler, K\"ahler) Lie algebras, as illustrated  in Section~\ref{scLa}.
We note that  by a  strictly almost K\"ahler (respectively,  strictly almost coK\"ahler) Lie algebra, we mean a Lie algebra which only admits non-integrable almost K\"ahler (respectively, non-normal almost coK\"ahler) structures.

\bigskip

Four-dimensional {non-abelian} symplectic Lie algebras have been classified up to isomorphisms by Ovando in \cite[Proposition~2.4]{Ovando1}. For each of these four-dimensional symplectic Lie algebras $(\h,\Omega)$, we can determine their derivations $D\in Der(\h)$ such that {$D+\alpha I$ is an infinitesimal symplectic transformation.} 

Let now  $(\g,\eta,\omega)$ be any five-dimensional cosymplectic Lie algebra. Then, by Theorem~\ref{teocosy}, it determines a corresponding four-dimensional symplectic Lie algebra, which must be isomorphic to one of the examples classified in \cite[Proposition~2.4]{Ovando1}. Consequently, by Theorem~\ref{isoalmostcosy}, $(\g,\eta,\omega)$ is  isomorphic to the cosymplectic structure arising from such example. Therefore, we have the following.

\begin{theorem} 
Let $(\g,\eta,\omega)$ be a five-dimensional cosymplectic Lie algebra. Then, $(g,\eta,\omega)$ is isomorphic to one of cosymplectic Lie algebras arising, as in Theorem~{\em\ref{teocosy}}, from one of the triples $(\h,\Omega,D)$, {where either $\h$ is abelian or  $(\h,\Omega)$ is one of non-abelian four-dimensional symplectic Lie algebras  listed in the following Table~$1$,} and $D$ a suitable derivation of $\h$. 
\end{theorem} 

\medskip
{\small
\begin{center}
\begin{tabular}{|c|c|l|}
\hline $\vphantom{\displaystyle\frac{A^{A^{A^{A}}}}{A}}$ $\h$ $\vphantom{\displaystyle\frac{A^{A^{A^{A}}}}{A}}$ &  Symplectic $2$-form $\Omega$&  { \qquad \qquad \qquad \qquad \quad  $D=\ad _{\xi}|_{\h}$ i.s.t.} \\
\hline  $ {\mathfrak {rh}} _3$  & $ e^{14}+e^{23}$  &   
$ \begin{array}{l}
\vphantom{\displaystyle\frac{A^{A^{A}}}{A}}
D e_1= 2p e_1 + q e_3 +r e_4, \quad
D e_2= s e_1 -p e_2 + t e_3 +q e_4, \\[4pt]
D e_3= p e_3, \quad
D e_4= -s e_3 - 2p e_4 \\[4pt]
\end{array}$
\\
\hline  $\mathfrak {rr} _{3,0}$     & $ e^{12}+e^{34} $  &  
$\begin{array}{l}
\vphantom{\displaystyle\frac{A^{A^{A}}}{A}}
D e_1= p e_2, \quad 
D e_3= q e_3 + r e_4,  \quad 
D e_4= s e_3 - q e_4 \\[4pt]
\end{array}$
 \\
\hline  $\vphantom{\displaystyle\frac{A^{A^{A}}}{A}}$ $\mathfrak {rr} _{3,-1}$     & $ e^{14}+e^{23} $  & 
$\begin{array}{l}
\vphantom{\displaystyle\frac{A^{A^{A}}}{A}}
D e_1= p e_4, \quad
D e_2= q e_2, \quad
D e_3= -q e_3 \\[4pt]
\end{array}$ \\
\hline  $\vphantom{\displaystyle\frac{A^{A^{A}}}{A}}$ $\mathfrak {rr'} _{3,0}$     & $ e^{14}+ e^{23} $ &   
$\begin{array}{l}
\vphantom{\displaystyle\frac{A^{A^{A}}}{A}}
D e_1= p e_4, \quad
D e_2= q e_3, \quad
D e_3= -q e_2 \\[4pt]
\end{array}$ \\
\hline  $\vphantom{\displaystyle\frac{A^{A^{A}}}{A}}$ $\mathfrak {r}_{2}\mathfrak r_2$     & {$e^{12}+\lambda e^{13}+ e^{34}, \  \lambda \geq 0$}  & 
$\begin{array}{l}
\vphantom{\displaystyle\frac{A^{A^{A}}}{A}}
D e_1= p e_2, \quad 
D e_3= q e_4 \\[4pt]
\end{array}$ \\
\hline  $\vphantom{\displaystyle\frac{A^{A^{A}}}{A}}$ $\mathfrak {r'}_{2}$     & $ e^{14}+ e^{23}$  & 
$\begin{array}{l}
\vphantom{\displaystyle\frac{A^{A^{A}}}{A}}
D e_1= p e_3 +q e_4, \quad
D e_2= -q e_3 + p e_4 \\[4pt]
\end{array}$ \\
\hline  $\vphantom{\displaystyle\frac{A^{A^{A}}}{A}}$ ${\mathfrak n}_4$  & $ e^{12}+e^{34}$  &  
$\begin{array}{l}
\vphantom{\displaystyle\frac{A^{A^{A}}}{A}}
D e_1= p e_2 +q e_3, \quad
D e_2= p e_3,\quad 
D e_4= p e_1 -q e_2 +r e_3 \\[4pt]
\end{array}$ \\
\hline   $\vphantom{\displaystyle\frac{A^{A^{A}}}{A}}$${\mathfrak r}_{4,0}$  & {$e^{14}+ \varepsilon e^{23}, \ \varepsilon = \pm 1$}  & 
$\begin{array}{l}
\vphantom{\displaystyle\frac{A^{A^{A}}}{A}}
D e_3= p e_2, \quad
D e_4= q e_1 \\[4pt]
\end{array}$ \\
\hline
\end{tabular} 
\\ \nopagebreak{$\vphantom{displaystyle{A^{A^A}}}$ {Table 1. $4$D symplectic Lie algebras giving rise to $5$D cosymplectic Lie algebras (I)} $\vphantom{\displaystyle\frac{a}{2}}$}
\end{center}
}

{\small
\begin{center}
\begin{tabular}{|c|c|l|}
\hline $\vphantom{\displaystyle\frac{A^{A^{A^{A}}}}{A}}$ $\h$ $\vphantom{\displaystyle\frac{A^{A^{A^{A}}}}{A}}$ &  Symplectic $2$-form $\Omega$&  { \qquad \qquad \qquad $D=\ad _{\xi}|_{\h}$ i.s.t.} \\
\hline  $\vphantom{\displaystyle\frac{A^{A^{A}}}{A}}$ ${\mathfrak r}_{4,-1}$  & $ e^{13}+e^{24} $  &  
$\begin{array}{l}
\vphantom{\displaystyle\frac{A^{A^{A}}}{A}}
D e_3= p e_2, \quad
D e_4= p e_1 +q e_2 \\[4pt]
\end{array}$ \\
\hline  $\vphantom{\displaystyle\frac{A}{A}}$ $\begin{array}{c}{\mathfrak r}_{4,-1,\beta} \\[6pt]   {-1\leq \beta<0} \end{array}$  & $ e^{12}+e^{34} $  &  
$\begin{array}{ll}
\begin{array}{l}
\vphantom{\displaystyle\frac{A^{A^{A}}}{A}}
D e_1= p e_1 \\[4pt]
D e_2= -p e_2 +\varepsilon q e_3 \\[4pt]
D e_4= \varepsilon q e_1 + r e_3 \\[4pt]
\end{array},\ 
\begin{array} {l}
\varepsilon= 0 {\textrm {\ if\ }} q \neq -1 \\[4pt]
\varepsilon= 1 {\textrm {\ if\ }} q =-1
\end{array}
\end{array}
$ \\
\hline   $\vphantom{\displaystyle\frac{A^{A^{A}}}{A}}$$\begin{array}{c}{\mathfrak r}_{4,\bar \alpha,-\bar \alpha} \\[4pt]  -1<\bar \alpha<0 \end{array}$  & $ e^{14}+e^{23} $  &
$\begin{array}{l}
\vphantom{\displaystyle\frac{A^{A^{A}}}{A}}
D e_2= p e_2, \quad
D e_3= -p e_3, \quad
D e_4=  q e_1  \\[4pt]
\end{array}$\\
\hline  $\vphantom{\displaystyle\frac{A^{A^{A^A}}}{A}}$ {$\mathfrak {r'} _{4,0,\delta}, \  \delta>0 $  }   & 
{$ e^{14}+ \varepsilon e^{23}, \ \varepsilon =\pm 1 $} &  
$\begin{array}{l}
\vphantom{\displaystyle\frac{A^{A^{A}}}{A}}
D e_2= p e_3, \quad
D e_3= -p e_2, \quad
D e_4=  q e_1 \\[4pt]
\end{array}$\\
\hline  $\vphantom{\displaystyle\frac{A^{A^{A}}}{A}}$ $\mathfrak {d} _{4,1}$   &  $\begin{array}{c}  e^{12}-e^{34} + \varepsilon e^{24}  \\[4pt]  \varepsilon =0,1  \end{array}$ &  
$\begin{array}{l}
\vphantom{\displaystyle\frac{A^{A^{A}}}{A}}
D e_1= (1-\varepsilon)p e_1, \quad
D e_2= (\varepsilon -1 )p e_2 + q e_3 \\[4pt]
D e_4= -q e_1 + r e_3 \\[4pt]
\end{array}$\\
\hline  $\vphantom{\displaystyle\frac{A^{A^{A}}}{A}}$ $\mathfrak {d} _{4,2}$   &  $\vphantom{\displaystyle\frac{A^{A^{A^{A^A}}}}{A}}$ 
{$\begin{array}{c} \\ e^{12}-e^{34} \\[10pt]  e^{14}+ \varepsilon e^{23},\ \varepsilon=\pm 1 \end{array}$ } &  
$ \begin{array} {l}
\begin{array}{l} \\
D e_1=  p e_1, \quad
D e_2= -p e_2, \quad
D e_4= q e_3 \\[4pt]
\end{array} 
\\
\begin{array}{l}
\vphantom{\displaystyle\frac{A^{A^{A}}}{A}}
D e_2=  p e_3, \quad
D e_4= -2p e_1 \\[4pt]
\end{array} \\[4pt]
\end{array}$ \\
\hline  
$ \begin{array}{c}\mathfrak {d} _{4,\lambda}\\[4pt]   \lambda\geq 1/2, \neq 1,2 \end{array}$   &   $ e^{12}-e^{34} $ & 
$ \begin{array} {ll} \begin{array}{l}
\vphantom{\displaystyle\frac{A^{A^{A}}}{A}}
D e_1=  p e_1 + \varepsilon q e_2  \\[4pt]
D e_2=  \varepsilon r e_1-p e_2  \\[4pt]
D e_4= s e_3 \\[4pt]
\end{array}, \ 
\begin{array} {c}
\varepsilon= 0 {\textrm {\ if\ }} \lambda \neq \frac{1}{2} \\[4pt]
\varepsilon= 1 {\textrm {\ if\ }} \lambda= \frac{1}{2}
\end{array}
\end{array}
$ \\
\hline  $\vphantom{\displaystyle\frac{A^{A^{A^A}}}{A}}$ 
{$\mathfrak {d'} _{4,\delta}, \  \delta>0 $  }  & {$\varepsilon( e^{12}-\delta e^{34}), \ 
\varepsilon = \pm 1 $}    &  
$\begin{array}{l}
\vphantom{\displaystyle\frac{A^{A^{A}}}{A}}
D e_1=  p e_2, \quad
D e_2=  -p e_1, \quad
D e_4= q e_3 \\[4pt]
\end{array}$ \\
\hline  $\vphantom{\displaystyle\frac{A^{A^{A}}}{A}}$ $ \mathfrak {h} _4 $    &  {$\varepsilon(e^{12}- e^{34}),  \
 \varepsilon = \pm 1 $ }    & 
$D e_2=  p e_1, \quad
D e_4= q e_3$  \\
\hline
\end{tabular}
\\ \nopagebreak { $\vphantom{displaystyle{A^{A^A}}}$ {Table 1. $4$D symplectic Lie algebras giving rise to $5$D cosymplectic Lie algebras (II)} $\vphantom{\displaystyle\frac{a}{2}}$}
\end{center}
}


%
We  now turn our attention to five-dimensional strictly almost coK\"ahler  Lie algebras. By Theorem~\ref{teoalmostcokah}, they arise from four-dimensional strictly almost K\"ahler Lie algebras, together with  a derivation $D$ such that {$D$ is an infinitesimal symplectic transformation}. Let  $\h$ denote  a four-dimensional strictly almost K\"ahler Lie algebra, that is, one admitting almost K\"ahler but not K\"ahler structures. By \cite{CF}, $\h$ is isomorphic to one of the following:
\begin{align} \label{akCF}
\begin{array}{l}
\mathfrak{rh}_3,\quad  	\mathfrak{r}'_2,\quad	\mathfrak{rr}_{3,-1},\quad	\mathfrak{n}_4,\quad	\mathfrak{r}_{4,0},\quad	\mathfrak{r}_{4,-1},\quad\mathfrak{r}_{4,-1,  \beta} \ (-1\leq  \beta<0), \\[4pt]
\mathfrak{r}_{4,\bar \alpha,-\bar \alpha} \ (-1<\bar \alpha<0), \quad \mathfrak{d}_{4,1},\quad	\mathfrak{d}_{4,\lambda} \ (\lambda > \frac{1}{2},\ \lambda \neq 1,2),\quad	\mathfrak{h}_4.
\end{array}
\end{align}

\noindent
For each of strictly almost K\"ahler Lie algebras listed  in \eqref{akCF}, the corresponding symplectic $2$-form $\Omega$ and {i.s.t. derivations $D$} are the ones listed in the previous Table~1.  By Theorem~\ref{isoalmostcokah}, each five-dimensional strictly almost coK\"ahler Lie algebra is then isomorphic to an almost coK\"ahler Lie algebra arising from one of such strictly almost K\"ahler examples.  Moreover, we observe that each of the above Lie algebras admits a family of  compatible almost complex structures $J$, which depends on at least four parameters. Thus, we get the following. 

\begin{theorem}
Let $(\g,\varphi,\xi,\eta,g)$ be a five-dimensional strictly almost coK\"ahler  Lie algebra. Then, $(\g,\varphi,\xi,\eta,g)$ is isomorphic to one of almost  coK\"ahler Lie algebras arising, as in Theorem~{\em\ref{teocokah}}, from one of the {quadruples $(\h,\Omega,J,D)$,} where  
\begin{itemize}
\item[(a)]  $\h$ is  one of Lie algebras listed in \eqref{akCF};
\vspace{4pt}\item[(b)]  its symplectic $2$-form $\Omega$ and {derivation} $D$ are as listed in Table~1;
\vspace{4pt}\item[(c)] {$J$ is a compatible almost complex structure.}
\end{itemize}
\end{theorem}


Comparing Theorems~\ref{teoalmostcokah} and \ref{teoKcosy}, one can find examples of {almost $\alpha$-coK\"ahler Lie algebras with $\mathcal L_\xi \varphi \neq 0$,}  considering  strictly almost K\"ahler Lie algebras $(\h,J, h)$ (equivalently, $(\h,\Omega,J)$) {together with a derivation $D$ such that $D+\alpha I$ is an i.s.t. but}  $DJ\neq JD$. We report below some explicit examples.



\begin{example}{\em {\bf  Almost coK\"ahler not $K$-cosymplectic Lie algebras from $\h_4$.}
We consider the following four-dimensional symplectic non-K\"ahler Lie algebra $(\h,\Omega)$ \cite{CF} and choose on it a compatible almost complex structure $J$:
$$\begin{array}{ll}
\h:=\h_4 : &  [e_1,e_2] = e_3, \quad [e_4,e_3] = e_3, \quad [e_4,e_1] = \frac{1}{2}e_1, \quad [e_4,e_2] = e_1 +\frac{1}{2}e_2, \\[6pt] 
& \Omega= \varepsilon (e^{12}-e^{34}), \quad \varepsilon = \pm 1, \qquad J e_1=- e_2, \quad J e_3=e_4. 
\end{array}
$$
Standard calculations yield that  i.s.t. $D \in Der(\h)$ are completely described by
$$D e_1= D e_3=0, \quad D e_2= p e_1,\quad D e_4=q e_3 . $$
\noindent
Moreover, for  $(p,q)\neq(0,0)$ we get nontrivial derivations which do not satisfy $DJ=JD$.  Consequently, applying Theorems~\ref{teoalmostcokah} and \ref{teoKcosy}, we conclude that $(\h,\Omega,J)$ gives rise to a family of five-dimensional strictly almost coK\"ahler (not $K$-cosymplectic) Lie algebras $(\g,\varphi, \xi, \eta,g)$, depending on two parameters  $(p,q)\neq(0,0)$,  explicitly described by 
$$\begin{array}{ll} \g : &[e_0,e_2]=p e_1,  \quad  [e_0,e_4]=q e_3, \\[6pt]
& [e_1,e_2] = e_3, \quad [e_4,e_3] = e_3, \quad [e_4,e_1] = \frac{1}{2}e_1, \quad [e_4,e_2] = e_1 +\frac{1}{2}e_2, \\[6pt]
& \xi=e_0, \quad \varphi(\xi)=0,\ \varphi|_{\h}=J,\quad \eta=e^0, \quad \Phi= \varepsilon(e^{12}-e^{34}),\ \varepsilon=\pm1.
\end{array}
$$
}\end{example}

\begin{example}{\em {\bf  Almost $\alpha$-coK\"ahler Lie algebras with $\mathcal L_\xi \varphi \neq 0$ from $\mathfrak r'_2$.}
We start with the following four-dimensional symplectic non-K\"ahler Lie algebra $(\h =\mathfrak r'_2,\Omega)$  (see also \cite{CF}) and choose on it a compatible almost complex structure $J$:
$$\begin{array}{ll}
\mathfrak r'_2 : & [e_1,e_3] = e_3, \quad [e_1,e_4] = e_4, \quad [e_2,e_3] = e_4, \quad [e_2,e_4] =- e_3,\\[6pt]
& \Omega=  e^{14}+e^{23},\qquad J e_3=e_2, \quad J e_4=e_1. 
\end{array}$$
By straightforward calculations we find that derivations $D \in Der(\h)$ such that $D+\alpha I$ is an i.s.t., are completely described by
\begin{align} \label{derr2}
	D e_1=p e_3+q e_4,\quad D e_2=-q e_3 +p e_4, \quad D e_3= -2\alpha e_3,\quad  D e_4=-2\alpha e_4.
\end{align}
\noindent
We note that if $\alpha \neq 0$ then for any $p,q \in \mathbb R$ the above are nontrivial derivations, for which  $DJ \neq JD$. Moreover,  if $\alpha=0$ then the above are nontrivial derivations for which $DJ \neq JD$ for any $(p,q)\neq(0,0)$. Therefore, applying Theorems~\ref{teoalmostcokah} and \ref{teoKcosy}, we conclude that $(\h,\Omega,J)$ gives rise to a family of five-dimensional strictly almost $\alpha$-coK\"ahler  Lie algebras $(\g,\varphi, \xi, \eta,g)$ with $\mathcal L_\xi \varphi \neq 0$, depending on two parameters   $p,q$ ($\neq(0,0)$ if $\alpha=0$),  explicitly described by 
$$\begin{array}{ll} 
\g :&[e_0,e_1]=p e_3 +q e_4, \quad [e_0,e_2]=-qe_3 +p e_4, \quad [e_0,e_3]=-2 \alpha e_3, \quad  [e_0,e_4]=-2\alpha e_4 \\[6pt]
 &[e_1,e_2] = e_3, \quad [e_4,e_3] = e_3, \quad [e_4,e_1] = \frac{1}{2}e_1, \quad [e_4,e_2] = e_1 +\frac{1}{2}e_2, \\[6pt]
 &\xi=e_0, \quad \varphi(\xi)=0,\ \varphi|_{\h}=J,\quad \eta=e^0, \quad \Phi= e^{14}+e^{23}.
\end{array}$$
}\end{example}


%

\medskip
Applying the same argument, we can use Theorem~\ref{isoKcosy} to classify five-dimension\-al strictly $K$-cosymplectic  Lie algebras, starting again  from the non-isomorphic examples of four-dimensional strictly almost K\"ahler Lie algebras $(\h,J,h)$ (equivalently, $(\h, \Omega,J)$) as classified in \cite{CF}, and considering their  i.s.t. $D\in Der(\h)$ such that $D$ commutes with $J$ (Theorem~\ref{teoKcosy}). 
In this way, we prove the following. 

\begin{theorem}\label{TeoKco}
Let $(\g,\varphi,\xi,\eta,g)$ be a five-dimensional strictly $K$-cosymplectic Lie  algebra. Then, $(\g,\varphi,\xi,\eta,g)$ is isomorphic to one of Lie algebras arising, as in Theorem~{\em\ref{teoKcosy}}, from  one of the almost K\"ahler Lie algebras listed in \eqref{akCF}, together with their symplectic $2$-form $\Omega$ as described in Table~$1$, any almost complex structures $J$ compatible with $\Omega$, and the trivial derivation $D=0$. 
\end{theorem} 

Finally, by Theorem~\ref{teocokah}, we know that five-dimensional coK\"ahler Lie algebras can be constructed starting from four-dimensional K\"ahler Lie algebras $(\h,J,h)$ admitting an i.s.t. $D\in Der(\h)$ satisfying $DJ=JD$. The classification of non-isomorphic {(non-abelian)} four-dimensional pseudo-K\"ahler Lie algebras was given  in \cite[Proposition 3.3]{Ovando2}. Among them, the ones admitting a definite positive inner product (and so, K\"ahler structures) are classified in \cite[Table 2]{CF}.  
Using this result and Theorem~\ref{isoalmostcokah}, we prove the following. 

\begin{theorem} 
Let $(\g,\varphi,\xi,\eta,g)$ be a five-dimensional coK\"ahler Lie algebra. {If $\ker \eta$ is not abelian, then} $(\g,\varphi,\xi,\eta,g)$ is isomorphic to one of coK\"ahler Lie algebras arising, as in Theorem~{\em\ref{teocokah}}, from one of the quadruples $(\h,\Omega,J,D)$ listed in the following Table~$2$, where $(\h, \Omega,J)$ is a four-dimensional K\"ahler  Lie algebra and $D$ a suitable derivation of $\h$.
\end{theorem}

{\footnotesize
\begin{center}
\begin{tabular}{|c|c|c|c|}
\hline $\h$  &  Symplectic $2$-forms $\Omega$ & $J=\varphi|_\h$   &$\begin{array}{l} D=\ad _{\xi}|_{\h} {\rm \ i.s.t.\ } \vphantom{A^{A^{A^A}}}\\[2pt] \text{with } DJ=JD\end{array}$\\
\hline $\vphantom{\displaystyle\frac{A^{A}}{A}}$  $\mathfrak {rr} _{3,0}$     & $\begin{array}{l}  a_{12}e^{12}+a_{34}e^{34},\  a_{12},a_{34}< 0 \end{array}$  &  $J e_1=e_2,\ Je_3=e_4$ & 
$\begin{array}{l}
De_3=-p e_4, \ De_4=p e_3 
\end{array}
$ \\
\hline  $\vphantom{\displaystyle\frac{A^{A}}{A}}$ $\mathfrak {rr'} _{3,0}$     & $\begin{array}{l}  a_{14}e^{14}+a_{23}e^{23},\  a_{14},a_{23} <0 \end{array}$ &  $J e_1=e_4,\ Je_2=e_3$  & 
$\begin{array}{l} 
De_2=-p e_3, \ De_3=p e_2 
\end{array}
$ \\
\hline  $\vphantom{\displaystyle\frac{A^{A}}{A}}$ $\mathfrak {r}_{2}\mathfrak r_2$     & $\begin{array}{l}  a_{12}e^{12}+a_{34}e^{34},\  a_{12}, a_{34} < 0 \end{array}$  &  $J e_1=e_2,\ Je_3=e_4$ & 
\ $D=0$ \\
\hline  $\vphantom{\displaystyle\frac{A^{A^{A^{A^A}}}}{A}}$ $\begin{array}{c}\mathfrak {r'} _{4,0,\delta}\\[4pt] \delta>0 \end{array}$     & $\begin{array}{l}  a_{14}e^{14}+a_{23}e^{23}, \  a_{14}>0,a_{23}< 0 \\[4pt] a_{14}e^{14}+a_{23}e^{23},\  a_{14},a_{23}> 0 \end{array}$ &  $\begin{array}{c} Je_4=e_1,\ Je_2= e_3 \\[4pt] Je_4=e_1,\ Je_3= e_2\end{array}$ &  
$\begin{array}{l} 
De_2=-p e_3, \ De_3=p e_2
\end{array}
$ \\
\hline $\vphantom{\displaystyle\frac{A^{A}}{A}}$   $\mathfrak {d} _{4,2}$   &   $a_{14}e^{14}+a_{23}e^{23},\  a_{14}, a_{23}< 0$  &   $Je_1=e_4,\ Je_2= e_3 $   & $ D=0$ \\
\hline $\vphantom{\displaystyle\frac{A^{A}}{A}}$ $\mathfrak {d} _{4,\frac{1}{2}}$ $\vphantom{\displaystyle\frac{A^{A}}{A}}$  &   $a_{12}(e^{12}-e^{34}),\  a_{12}<0$ &  $Je_1=e_2,\ Je_4=e_3$   & 
$\begin{array}{l} 
De_1=-p e_2, \ De_2=p e_1 
\end{array}
$ \\
\hline  $\vphantom{\displaystyle\frac{A^{A^{A^{A^A}}}}{A}}$ $\begin{array}{c} \mathfrak {d'} _{4,\delta} \\[4pt] \delta>0 \end{array} $  $\vphantom{\displaystyle\frac{A^{A^{A^{A^A}}}}{A}}$   &   $\begin{array}{c} a_{12}(e^{12}-\delta e^{34}),\  a_{12}<0\\[4pt] a_{12}(e^{12}-\delta e^{34}),\  a_{12}>0 \end{array}$   &  $\begin{array}{c} Je_1= e_2,\ Je_4=e_3 \\[4pt]  Je_2=e_1,\ Je_3=e_4\end{array}$  & 
$\begin{array}{l}
De_1=-p e_2, \ De_2=p e_1 
\end{array}
$ \\
\hline
\end{tabular}
\\ \nopagebreak { $\vphantom{displaystyle{A^{A^A}}}$ {Table 2. $4$D K\"ahler Lie algebras giving rise to $5$D coK\"ahler Lie algebras $\vphantom{\displaystyle\frac{a}{2}}$}}
\end{center}
}

\begin{example}{\em {\bf  $\alpha$-Kenmotsu Lie algebras.}
Consider the abelian  K\"ahler Lie algebra $(\h,\Omega,J)$, where $\h=span(e_1,e_2,e_3,e_4)$, $\Omega=  e^{31}+e^{42}$ and  $J e_1=e_3, \ J e_2= e_4$ is a compatible  complex structure.
%

By direct  calculations we find that derivations $D \in Der(\h)$ such that $D+\alpha I$ is an i.s.t. and $DJ=JD$, are of the form

$$
D=
\left(\begin{array}{rrrr}
 -\alpha &  p & q& r \\[4pt]
 -p &-\alpha & r & s \\[4pt]
 -q & -r &  -\alpha &  p\\[4pt]
 -r & -s &-p &-\alpha 
\end{array}
\right).
$$%
\noindent
Consequently, applying Theorem \ref{teocokah}, we conclude that from $(\h,\Omega,J)$ {one can construct} a family of five-dimensional $\alpha$-Kenmotsu  Lie algebras $(\g,\varphi, \xi, \eta,g)$, depending, for each given $\alpha \neq 0$, on four parameters: 
$$\begin{array}{ll} 
\g :&[e_0,e_1]=-\alpha e_1 -p e_2 -q e_3 -r e_4, \quad\,  [e_0,e_2]=p e_1 -\alpha e_2 -r e_3 -s e_4,  \\[6pt]
 &[e_0,e_3]= q e_1 + re_2 -\alpha e_3 -p e_4, \qquad  [e_0,e_4]=r e_1 +se_2+p e_3 -\alpha e_4,  \\[6pt]
 &\xi=e_0, \quad \varphi(\xi)=0,\ \varphi|_{\h}=J,\quad \eta=e^0, \quad \Phi= e^{31}+e^{42}.
\end{array}$$
}\end{example}

\section{Concluding remarks}
\setcounter{equation}{0}

A {\em Sasakian structure} on an odd-dimensional manifold $M$ is a normal contact metric structure $(\varphi,\xi,\eta,g)$, that is, a normal almost contact metric structure such that $d\eta=\Phi$. 

As such, clearly Sasakian structures are as far as possible from coK\"ahler structures, for which $d\eta=0$. Nevertheless, CoK\"ahler and Sasakian structures are both considered as odd-dimen\-sional counterparts to K\"ahler structures. Which one of them is the most natural analogue to K\"ahler structures, is actually the topic of debate among researchers. 

The results of this paper give a contribution to this debate, showing that with regard to left-invariant structures on Lie groups, the link between K\"ahler and coK\"ahler structures seems to be stronger and straightforward than the one between K\"ahler and Sasakian structures. In fact, any coK\"ahler Lie algebra arises from a K\"ahler structure together with a suitable derivation (Theorem~\ref{teocokah}). On the other hand, Sasakian structures can be built by contactization of a K\"ahler Lie algebra \cite{AFV}, and in particular, all Sasakian Lie algebras with nontrivial center can be built in this way, but there also exist Sasakian Lie algebras with trivial center, which do not arise from a corresponding K\"ahler Lie algebra.

The same fact happens for the wider corresponding classes of $K$-cosymplectic and {\em $K$-contact} Lie algebras (that is, contact metric Lie algebras whose characteristic vector field is Killing). In fact, we can build any $K$-cosymplectic Lie algebra from a corresponding almost K\"ahler Lie algebra with some suitable derivation (Theorem~\ref{teoKcosy}), while there exist $K$-contact Lie algebras (the ones with trivial center) which do not arise from any almost K\"ahler Lie algebra \cite{CF}.

\end{document}